\documentclass[11pt]{article}
\usepackage{amsmath}
\usepackage{amsfonts}
\usepackage{amssymb}

\newtheorem{theorem}{Theorem}
\newtheorem{corollary}[theorem]{Corollary}

\newtheorem{lemma}[theorem]{Lemma}

\newtheorem{preremark}[theorem]{Remark}
\newenvironment{remark}    {\begin{preremark}\rm}{\hfill$\Diamond$\end{preremark}}
\newenvironment{proof}[1][Proof]{\noindent\textbf{#1.} }{\ \rule{0.5em}{0.5em}}

\numberwithin{equation}{section}
\numberwithin{theorem}{section}

\begin{document}

\title{Complex structures adapted to magnetic flows}
\author{Brian C. Hall \thanks{\emph{University of Notre Dame, 255 Hurley Building, Notre Dame, IN \ 46556 \ USA}\newline E-mail: bhall@nd.edu} \ and William D. Kirwin \thanks{\emph{Mathematics Institute, University of Cologne, Weyertal 86 - 90, 50931 Cologne, Germany}.\newline E-mail: will.kirwin@gmail.com}}
\date{}
\maketitle

\begin{abstract}
Let $M$ be a compact real-analytic manifold, equipped with a real-analytic Riemannian metric $g,$ and let $\beta$ be a closed real-analytic $2$-form on $M$, interpreted as a magnetic field. Consider the Hamiltonian flow on $T^{\ast}M$ that describes a charged particle moving in the magnetic field $\beta$. Following an idea of T. Thiemann, we construct a complex structure on a tube inside $T^{\ast}M$ by pushing forward the vertical polarization by the Hamiltonian flow ``evaluated at time $i$.'' This complex structure fits together with $\omega-\pi^{\ast}\beta$ to give a K\"{a}hler structure on a tube inside $T^{\ast}M$. We describe this magnetic complex structure in terms of its $(1,0)$-tangent bundle, at the level of holomorphic functions, and via a construction using the embeddings of Whitney--Bruhat and Grauert, which is a magnetic analogue to the analytic continuation of the geometric exponential map. We describe an antiholomorphic intertwiner between this complex structure and the complex structure induced by $-\beta$, and we give two formulas for local K\"{a}hler potentials, which depend on a local choice of vector potential $1$-form for $\beta$. When $\beta=0$, our magnetic complex structure is the adapted complex structure of Lempert--Sz\H{o}ke and Guillemin--Stenzel.

We compute the magnetic complex structure explicitly for constant magnetic fields on $\mathbb{R}^{2}$ and $S^{2}.$ In the $\mathbb{R}^{2}$ case, the magnetic adapted complex structure for a constant magnetic field is related to work of Kr\"otz--Thangavelu--Xu on heat kernel analysis on the Heisenberg group.

\end{abstract}

\section{Introduction}

Adapted complex structures were introduced, independently and in different but equivalent ways, by L. Lempert and R. Sz\H{o}ke \cite{Lempert-Szoke, Szoke} and by V. Guillemin and M. Stenzel \cite{Guillemin-Stenzel1,Guillemin-Stenzel2}. Let $(M,g)$ be a real-analytic Riemannian manifold\footnote{When we say that the pair $(M,g)$ is real analytic, we always mean that $M$ is a real-analytic manifold and $g$ is a real-analytic metric on $M.$}, let $TM$ be the tangent bundle of $M$, and let $T^{R}M$ denote the ``tube'' of radius $T,$ that is, the set of vectors in $TM$ with length less than $R>0$. For each unit-speed geodesic $\gamma$ in $M$, we can define a map $\Psi_{\gamma}$ of the complex plane into $TM$ by setting
\[
\Psi_{\gamma}(\sigma+i\tau)=N_{\tau}\dot{\gamma}(\sigma)\in T_{\gamma(\sigma
)}M,
\]
where $N_{\tau}$ is scaling in the fibers by $\tau$. In the terminology of Lempert--Sz\H{o}ke \cite[Def 4.1]{Lempert-Szoke}, a complex structure on some $T^{R}M$ is called \emph{adapted} (to the metric on $M$) if, for each $\gamma$, the map $\Psi_{\gamma}$ is holomorphic as a map of the strip $\{(\sigma,\tau):\left\vert \tau\right\vert<R\}\subset\mathbb{C}$ into $T^{R}M$. On any $T^{R}M$, there is at most one adapted complex structure, and if $M$ is compact, then an adapted complex structure exists on $T^{R}M$ for all sufficiently small $R$. Guillemin and Stenzel define their complex structures on a tube in $T^{\ast}M$ in terms of a K\"{a}hler potential and an involution. It turns out that these complex structures are (after identifying tangent and cotangent bundles by means of the metric) precisely the same ones defined in the apparently different way by Lempert and Sz\H{o}ke.

For certain very special manifolds $M$, the adapted complex structure exists globally, that is, on all of $TM\cong T^{\ast}M.$ Examples of such manifolds include compact Lie groups with bi-invariant metrics, compact symmetric spaces, and the Gromoll--Meyer exotic $7$-sphere \cite{Aguilar01}. A necessary, but certainly not sufficient, condition for the adapted complex structure to be global is that $M$ should have nonnegative sectional curvature. When $M$ is a compact Lie group with a bi-invariant metric, nice results hold for the geometric quantization of $T^{\ast}M$ with the polarization coming from the adapted complex structure \cite{Hall02,Florentino-Matias-Mourao-Nunes05,Florentino-Matias-Mourao-Nunes06}.

Meanwhile, in \cite[Sec 2.1]{Thiemann96}, T. Thiemann proposes a ``complexifier'' method for introducing complex structures on cotangent bundles of manifolds. Let $C$ be a smooth function on $T^{\ast}M$ and let $X_{C}$ be the associated Hamiltonian vector field. Let $f$ be a function that is constant along the leaves of the cotangent bundle and define
\begin{equation}
f_{\mathbb{C}}=e^{iX_{C}}(f), \label{eqn:Thiemann}
\end{equation}
provided that this can be defined in some natural way either on all of $T^{\ast}M$ or on some portion thereof. (For real $\sigma$, $\exp(\sigma X_{C})f$ is just the composition of $f$ with the classical flow generated by $C$. To put $\sigma=i$, we need to attempt to analytically continue the expression $\exp(\sigma X_{C})f$ with respect to $\sigma$.) Thiemann proposes that the complex structure associated to the function $C$ is the one for which the holomorphic functions are precisely the functions of the form $f_{\mathbb{C}},$ where $f$ is constant along the fibers of $T^{\ast}M$.

For a general $C$ (even assumed to be real analytic), it is not clear to what extent one can carry out this program, because of convergence questions associated to the analytic continuation. Nevertheless, in the case that $M$ is a compact Lie group with bi-invariant metric, if we take $C$ to be half the length-squared in the fibers, then it is not hard to show that Thiemann's prescription makes sense and gives the adapted complex structure on $T^{\ast }M.$ (See Equation (3.37) in \cite{Thiemann1}, Equation (3.8) in \cite{Thiemann2}, and Section 4 of \cite{Hall02}.)

In \cite{Hall-K_acs}, we attempt to understand adapted complex structures from the point of view of Thiemann's complexifier method. We consider a compact Riemannian manifold and we take the complexifier to be the energy function $E$, equal to half the length-squared in the fibers. We show, in essence, that Thiemann's method, applied in this case, gives the adapted complex structure, but expressed initially in terms of the $(1,0)$-subspace rather than in terms of holomorphic functions. For real $\sigma$, $\exp(\sigma X_{E})f$ is simply the composition of $f$ with the geodesic flow $\Phi_{\sigma}$ at ``time'' $\sigma$. (The parameter $\sigma $ should not be interpreted as the physical time, but simply as a parameter.) For each point $z\in TM$, let $V_{z}$ be the complexification of the vertical subspace. Let $P_{z}(\sigma)$ denote the pushforward of $V_{\Phi_{-\sigma }(z)}$ by $\Phi_{\sigma}$. We show that for small enough $R$ and for each $z\in T^{R}M$, it is possible to analytically continue the map $\sigma\mapsto P_{z}(\sigma)$ to a disk of radius greater than one. We then show that the subspaces $P_{z}(i)$ are the $(1,0)$-subspace for a complex structure on $T^{R}M,$ and that this complex structure is adapted in the sense of Lempert--Sz\H{o}ke. From this point of view, we are able to give simple arguments for the known properties of the adapted complex structure, including the K\"{a}hler potential and involution of \cite{Guillemin-Stenzel1}. Furthermore, we show that every holomorphic function on $T^{R}M$ can be obtained by Thiemann's method (\ref{eqn:Thiemann}) with $C=E$, where $\exp(iX_{E})f$ is given by an absolutely convergent series. (Formulas similar to this appear in both \cite{Lempert-Szoke,Szoke} and \cite{Guillemin-Stenzel1,Guillemin-Stenzel2}.)

In the present paper, we apply Thiemann's method to construct a new family of complex structures on (co)tangent bundles of Riemannian manifolds, generalizing the adapted complex structure. Specifically, we consider a real-analytic Riemannian manifold together with a closed real-analytic $2$-form $\beta$ on $M$, where $\beta$ is interpreted as a magnetic field. The dynamics of a charged particle moving in this magnetic field may be described as follows. Consider the canonical $2$-form $\omega$ on $T^{\ast}M$ and subtract from this the pullback $\pi^{\ast}\beta$ of the magnetic $2$-form by the projection $\pi:T^{\ast}M\rightarrow M$, resulting in the ``twisted'' symplectic form $\omega^{\beta}.$ Let $\Phi _{\sigma}$ be the Hamiltonian flow associated to the usual energy function $E$ (half the length-squared in the fibers) with respect to the twisted symplectic form $\omega-\pi^{\ast}\beta$. In the interests of keeping the notation manageable, we suppress the dependence of flow on $\beta$. Locally, the approach we are using is equivalent to keeping the symplectic form equal to $\omega$ and adding a magnetic term to the energy function.

Following our approach in \cite{Hall-K_acs}, we define a family of subspaces $P_{z}(\sigma)$ by pushing forward the (complexified) vertical subspace by $\Phi_{\sigma}.$ We show that there is some $R>0$ for which the following hold. First, for all $z$ in a tube $T^{\ast,R}M$, the map $\sigma\mapsto P_{z}(\sigma)$ has an analytic continuation to a disk of radius greater than one. Second, the subspaces $P_{z}(i)$ are the $(1,0)$-subspaces for an integrable almost complex structure on $T^{\ast,R}M$. Third, this complex structure fits together with the symplectic form $\omega-\pi^{\ast}\beta$ to give a K\"{a}hler structure on $T^{\ast,R}M$. In addition, we give a local expression for a K\"{a}hler potential in terms of a locally defined $1$-form $A$ on $M$ with $dA=\beta$. In contrast to the case of adapted complex structures, inversion in the fibers (the maps sending each $p\in T_{x}^{\ast}M$ to $-p$) is not antiholomorphic, but rather antiholomorphically intertwines the complex structures associated to $\beta$ and $-\beta$.

In fact, $P_{z}(\sigma+i\tau)$ induces a K\"{a}hler structure on $T^{\ast,R}M$ for any $\sigma+i\tau\in D_{1+\varepsilon}$ with $\tau>0$, both in the magnetic and $\beta=0$ cases. In the $\beta=0$ case, the recent papers \cite{Lempert-Szoke10-1} and \cite{Lempert-Szoke10-2} give an alternate construction of this family of complex structures and study the induced family of K\"{a}hler quantizations of $T^{\ast,R}M$.

Although our main theorems are proved for the case where $M$ is compact, the definitions also make sense for noncompact $M$. Similar results should hold on some neighborhood of the zero-section, but such a neighborhood in the noncompact case does not necessarily contain a tube around the zero-section.

We also compute the magnetic complex structure for the cases of a constant magnetic field on the plane $\mathbb{R}^{2}$ and the sphere $S^{2}.$ In these cases, the complex structure can be computed explicitly and exists on the whole cotangent bundle. In the plane, we also compute explicitly a K\"{a}hler potential, and explain how it is related to the work of Kr\"{o}tz--Thangavelu--Xu \cite{Krotz-Thangavelu-Xu} on heat kernel analysis on Heisenberg groups.

\section{Main Results}

Let $(M,g)$ be a compact, real-analytic Riemannian manifold, let $T^{\ast}M$ be its cotangent bundle, and define
\[
T^{\ast,R}M=\left\{\left(x,p\right)\in T^{\ast}M\left\vert g(p,p)<R^2\right.\right\}.
\]
Let $\theta$ be the canonical 1-form on $T^{\ast}M$ and let $\omega:=-d\theta$ be the canonical 2-form. Now let $\beta$ be a closed, real-analytic 2-form on $M$ and let $\pi^{\ast}\beta$ be the pullback of $\beta$ by the projection $\pi:T^{\ast}M\rightarrow M.$ We then consider the ``twisted'' 2-form $\omega^{\beta}$ on $T^{\ast}M$ given by
\[
\omega^{\beta}=\omega-\pi^{\ast}\beta.
\]
It is easily verified that $\omega^{\beta}$ is nondegenerate and thus defines a symplectic form on $T^{\ast}M.$

We consider the energy function on $T^{\ast}M$ given by
\[
E(x,p)=\frac{1}{2}g(p,p)
\]
on $T^{\ast}M$ and we let $\Phi_{\sigma}$ denote the Hamiltonian flow associated to the symplectic manifold $(T^{\ast}M,\omega^{\beta})$ and the energy function $E.$ That is to say, $\Phi_{\sigma}$ is the flow along the Hamiltonian vector field $X_{E},$ where $X_{E}$ satisfies $\omega^{\beta}(X_{E},\cdot)=dE.$ If $\beta=0$, then $\Phi_{\sigma}$ is the geodesic flow on $T^{\ast}M$. For $\beta\neq0$, we may interpret $\Phi_{\sigma}$ as describing the motion of a charged particle moving on $M$ under the influence of the magnetic field $\beta$. We will refer to $\Phi_{\sigma}$ as the magnetic flow. In this context, a vector $p\in T_{x}^{\ast}M$ should be understood as the ``kinetic momentum,'' which is related in a simple way to the velocity of the particle. Because we are using the form $\omega^{\beta}$ rather than $\omega,$ the Poisson bracket of two components of the (kinetic) momentum is not zero but is expressed in terms of $\beta.$

In this paper, we construct a ``magnetic'' complex structure on some $T^{\ast,R}M$ by means of the ``imaginary-time flow'' $\Phi_{i},$ where $\Phi_{i}$ is defined by a suitable analytic continuation of $\Phi_{\sigma}$ with respect to $\sigma.$ (Here, ``time'' should not be understood in a physical sense, but rather simply as the parameter of a flow.)
Actually, we will give three different ways of understanding the magnetic complex structure on $T^{\ast,R}M$, each of which involves a different sort of analytic continuation of $\Phi_{\sigma}.$ We now briefly describe these three approaches.

First, for any $z$ in $T^{\ast}M,$ let $V_{z}\subset T_{z}^{\mathbb{C}} T^{\ast}M$ denote the complexification of the vertical tangent space to $T^{\ast}M$ at $z$. For any $\sigma\in\mathbb{R},$ let $P_{z}(\sigma)$ denote the pushforward of $V_{z}$ by the flow $\Phi_{\sigma}$:
\begin{equation}
P_{z}(\sigma)=\left(  \Phi_{\sigma}\right)  _{\ast}(V_{\Phi_{-\sigma}(z)})\subset T_{z}^{\mathbb{C}}T^{\ast}M. \label{eqn:mag-geod_flow}
\end{equation}
For each fixed $z,$ we will analytically continue the map $\sigma\mapsto P_{z}(\sigma)$ to a holomorphic map of a disk in $\mathbb{C}$ to the Grassmannian of complex Lagrangian subspaces of $T_{z}^{\mathbb{C}}T^{\ast}M.$ Second, if $f$ is any real-analytic function on $M,$ then $f\circ\pi$ is a function on $T^{\ast}M$ that is constant along the leaves of the vertical distribution. Thus, $f\circ\pi\circ\Phi_{\sigma}$ is constant in the directions of $P(-\sigma).$ For each fixed $z\in T^{\ast}M,$ we will analytically continue the map $\sigma\mapsto f\circ\pi\circ\Phi_{\sigma}(z).$ Third, suppose we embed $M$ in a real-analytic way into a complex manifold $X$ as a totally real submanifold of maximal dimension, as in the work of Bruhat--Whitney and Grauert. Then the map $\sigma\mapsto\pi\circ\Phi_{\sigma}(z)$ is a real-analytic map of $\mathbb{R}$ into $M\subset X.$ We will analytically continue this map to a holomorphic map of a disk in $\mathbb{C}$ into $X.$

The main results of this paper can then be described briefly as follows.

\begin{theorem}
\label{introMain.thm}There is some $R>0$ and a ``magnetic'' complex structure on $T^{\ast,R}M$ such that the
following results hold.

\begin{enumerate}
\item \label{main.vert}For all $z\in T^{\ast,R}M,$ the map $\sigma\mapsto P_{z}(\sigma)$ can be analytically continued to a disk about the origin of radius greater than 1 and $P_{z}(i)$ is the $(1,0)$-distribution of the magnetic complex structure.

\item \label{main.grauert}Suppose $X$ is a complex manifold with $M$ real analytically embedded into $X$ as a totally real submanifold of maximal dimension. Then for all $z\in T^{\ast,R}M,$ the map $\sigma\mapsto\pi\circ \Phi_{\sigma}(z)$ can be analytically continued to a disk about the origin of radius greater than 1. Furthermore, the map
\[
z\mapsto\pi\circ\Phi_{i}(z)
\]
is a diffeomorphism of $T^{\ast,R}M$ onto its image inside $X,$ and the pullback of the complex structure on $X$ by this map is the magnetic complex structure on $T^{\ast,R}M.$

\item \label{main.funct}For every real-analytic function $f$ on $M,$ there is some $r\leq R$ such that for each $z\in T^{\ast,r}M,$ the map $\sigma\mapsto f\circ\pi\circ\Phi_{\sigma}(z)$ can be analytically continued to a disk about the origin of radius greater than 1. Furthermore, the function the function $f_{\mathbb{C}}:T^{\ast,r}M\rightarrow\mathbb{C}$ given by
\[
f_{\mathbb{C}}=f\circ\pi\circ\Phi_{i}
\]
is holomorphic on $T^{\ast,r}M$ with respect to the magnetic complex structure.
\end{enumerate}
\end{theorem}

The three points in Theorem \ref{introMain.thm} are proved in Sections
\ref{sec:continuation}, \ref{grauert.sec}, and \ref{functions.sec},
respectively. If we take $X$ to be $T^{\ast,R}M$ itself with the magnetic
complex structure, then the map in Point \ref{main.grauert} becomes an
identity:
\[
\pi\circ\Phi_{i}(z)=z.
\]

In the nonmagnetic ($\beta=0$) case, the flow $\Phi_{\sigma}$ is the geodesic flow, which satisfies
\[
\pi(\Phi_{\sigma}(x,p))=\pi(\Phi_{1}(x,\sigma p))=\exp_{x}(\sigma p),\quad\text{(}\beta=0\text{ case),}
\]
where $\exp_{x}$ is (after identifying the tangent and cotangent spaces using the metric) the geometric exponential map. Thus, when $\beta=0,$ we may replace the map $(x,p)\mapsto\pi(\Phi_{i}(x,p))$ with the map
\begin{equation}
(x,p)\mapsto\exp_{x}(ip),\quad\text{(}\beta=0\text{ case).}
\label{betaZeroMap}
\end{equation}
In the $\beta=0$ case, the map (\ref{betaZeroMap}) first appears in the proof of Proposition 3.2 of \cite{Szoke} and is implicitly contained in Section 5 of \cite{Guillemin-Stenzel2}. See also Theorem 15 in \cite{Hall-K_acs} and the description of the adapted complex structure in Section 1.1 of \cite{Zelditch_Grauert1}. When $\beta\neq0,$ the expressions $\pi(\Phi_{i}(x,p))$ and $\pi(\Phi_{1}(x,ip))$ are no longer equal. We follow the philosophy of Thiemann's complexifier method by putting the analytic continuation into the time-parameter of the flow, as in Point \ref{main.grauert} in Theorem \ref{introMain.thm}.

To construct the magnetic complex structure, we first construct (in Section \ref{sec:continuation}) an analytic continuation of the flow $\Phi_{\sigma}$ in local coordinates, using elementary techniques from differential equations in a complex domain. We then use the analytically continued flow to construct the subspaces $P_{z}(i)$, as in Point \ref{main.vert} of the theorem, on some tube. We next verify that $z\mapsto P_{z}(i)$ is an involutive complex distribution and that (on some possibly smaller tube), $P_{z}(i)$ is disjoint from its complex conjugate. This establishes the existence of a complex structure satisfying Point \ref{main.vert} of the theorem. We then verify (in Sections \ref{grauert.sec} and \ref{functions.sec}) that this complex structure satisfies Points \ref{main.grauert} and \ref{main.funct} of Theorem \ref{introMain.thm}.

Most of the proofs are either ``by continuity''---meaning that a certain property can be verified directly on the zero-section and thus holds also in a neighborhood of the zero-section---or ``by analyticity''---meaning that a certain property holds for the flow $\Phi_{\sigma}$ at real times and thus also at imaginary times. For example, the distribution $z\mapsto P_{z}(\sigma)$ is certainly integrable for real values of $\sigma,$ since it is the pushforward of an integrable distribution by a diffeomorphism of $T^{\ast}M.$ It is then not hard to show that $z\mapsto P_{z}(\sigma+i\tau)$ is integrable for complex numbers $\sigma+i\tau.$

\begin{theorem}
The ``magnetic'' complex structure in Theorem \ref{introMain.thm} has the following additional properties.

\begin{enumerate}
\item The magnetic complex structure fits together with the symplectic form $\omega^{\beta}$ to give a K\"{a}hler structure on $T^{\ast,R}M.$

\item The zero-section $M\subset T^{\ast,R}M$ is a totally real submanifold of $T^{\ast,R}M$ with respect to the magnetic complex structure.

\item The map $(x,p)\mapsto(x,-p)$ is an antiholomorphic map between $\ T^{\ast,R}M$ with the magnetic complex structure associated to $\beta$ and $T^{\ast,R}M$ with the magnetic complex structure associated to $-\beta.$
\end{enumerate}
\end{theorem}

In Section \ref{sec:K}, we construct a local K\"{a}hler potential for the magnetic complex structure. Given a point $x_{0}\in M,$ we can find an open set $U\subset M$ containing $x_{0}$ and a real-analytic 1-form $A$ defined on $U$ such that $dA=\beta.$ Now, for each $z\in\pi^{-1}(U)$ and all sufficiently small $\sigma\in\mathbb{R},$ let $f_{\sigma}(z)$ denote the real number given by
\[
f_{\sigma}(z)=\sigma E(z)+\int_{-\sigma}^{0}A\left(\frac{d(\pi\circ\Phi_{s}(z))}{ds}\right)~ds.
\]
The assumption that $\sigma$ is small guarantees that $\pi\circ\Phi_{s}(z)$ remains in $U$ for all $s\in\lbrack0,\sigma].$

\begin{theorem}
There is a subneighborhood $V\subset U$ of $x_{0}$ and an $r\in(0,R]$ such that for each $z=(x,p)$ with $x\in V$ and $\left\vert p\right\vert <r,$ the map
\[
\sigma\mapsto f_{\sigma}(z)
\]
can be analytically continued to a disk of radius greater than 1 around the origin in $\mathbb{C}.$ The function $f_{-i}$ defined by this analytic continuation satisfies
\[
Xf_{-i}=\theta(X)+\pi^{\ast}A(X)
\]
for each vector $X$ that is of type $(0,1)$ with respect to the magnetic complex structure.
\end{theorem}

\begin{corollary}
Let $L\rightarrow T^{\ast,R}M$ be a hermitian line bundle with compatible connection having curvature $-i\omega^{\beta},$ equipped with the holomorphic structure induced by the magnetic complex structure on $T^{\ast,R}M$. Let $L^{\otimes k}$ be the $k$th tensor power of $L.$ Then in the local trivialization of $L^{\otimes k}$ determined by $\theta^{A}$, the function $\exp\{-ikf_{-i}(z)\}$ is a local holomorphic section.
\end{corollary}

\begin{corollary}
The function
\[
-i(f_{i}-f_{-i})
\]
is real valued and a (local) K\"{a}hler potential.
\end{corollary}

The last two sections of the paper compute the magnetic complex structure for the case of ``constant'' magnetic fields on $\mathbb{R}^{2}$ or $S^{2}.$

\begin{theorem}
If $\beta$ is a constant multiple of the area form on $\mathbb{R}^{2},$ then the magnetic complex structure is defined on all of $T^{\ast}\mathbb{R}^{2}$, and $T^{\ast}\mathbb{R}^{2}$ with the magnetic complex structure is biholomorphic to $\mathbb{C}^{2}.$

If $\beta$ is a constant multiple of the area form on the sphere $S^{2}$ of radius $r,$ then the magnetic complex structure is defined on all of $T^{\ast}S^{2}$, and $T^{\ast}S^{2}$ with the magnetic complex structure is biholomorphic to $S_{\mathbb{C}}^{2},$ where
\[
S_{\mathbb{C}}^{2}=\left\{  \left.  \mathbf{a}\in\mathbb{C}^{3}\right\vert a_{1}^{2}+a_{2}^{2}+a_{3}^{2}=r^{2}\right\}.
\]
\end{theorem}

We give explicit formulas for the map from $T^{\ast}\mathbb{R}^{2}$ to $\mathbb{C}^{2}$ and the map from $T^{\ast}S^{2}$ to $S_{\mathbb{C}}^{2}.$ Both maps depend on \textit{which} multiple of the area form is used. In particular, the maps for nonzero multiples of the area form do not coincide with the maps for the $\beta=0$ case, which is just the case of the ordinary adapted complex structure. In the $\mathbb{R}^{2}$ case, we also compute a global K\"{a}hler potential, which is related to formulas appearing in \cite{Krotz-Thangavelu-Xu}.

\section{Analytic continuation of the magnetic flow \label{sec:continuation}}

Throughout the paper, $(M^{n},g)$ will denote a compact, real-analytic Riemannian manifold (meaning that both $M$ and $g$ are real analytic), $\beta$ will denote a closed, real-analytic 2-form on $M,$ and $\Phi_{\sigma}$ will denote the Hamiltonian flow on $T^{\ast}M$ associated to the energy function $E(x,p)=\frac{1}{2}g(p,p)$ and the symplectic form $\omega^{\beta}:=\omega-\pi^{\ast}\beta.$ In this section, we construct an analytic continuation of the flow $\Phi_{\sigma}$ in each local coordinate system, and use this to construct the subspaces $P_{z}(i)$ on some tube. We then verify that $P_{z}(i)$ is the $(1,0)$-distribution of some complex structure on $T^{\ast,R}M.$

We use the summation convention throughout the paper. Given local coordinates $(x^{1},\ldots,x^{n})$ on $M,$ we represent the 2-form $\beta$ at a point $x\in M$ by the matrix $\beta_{jk}(x)$ given by
\[
\beta_{jk}(x)=\beta\left(\frac{\partial}{\partial x^{j}},\frac{\partial}{\partial x^{k}}\right).
\]
Then $\beta$ can be expressed as $\beta=\frac{1}{2}\beta_{jk}dx^{j}\wedge dx^{k}.$ We also use the standard local coordinates $(x^{1},\ldots,x^{n},p_{1},\ldots,p_{n})$ on $T^{\ast}M,$ for which the point in $T^{\ast}M$ is $p_{j}dx^{j}.$ Note that since $dE=0$ on the zero-section, we have $\Phi_{\sigma}((x,0))=(x,0)$ for all $x\in M.$

\begin{theorem}
\label{thm:zero-section}(Behavior on the zero-section) For each point $(x,0)$ in the zero-section, let us choose coordinates $x^{1},\ldots,x^{n}$ in a neighborhood of $x$ so that the metric tensor at $x$ is the identity. Then $(\Phi_{\sigma})_{\ast}$ is the linear map of $T_{(x,0)}T^{\ast}M$ to itself represented in local coordinates by the matrix
\[
\begin{pmatrix}
1 & \sigma\frac{\exp(\sigma\beta)-\mathbf{1}}{\sigma\beta}\\
0 & \exp(\sigma\beta)
\end{pmatrix}
.
\]
It follows that the pushforward of the complexified vertical subspace is the space of vectors representable in local coordinates as
\begin{equation}
\begin{pmatrix}
\sigma\frac{\exp(\sigma\beta)-\mathbf{1}}{\sigma\beta}\mathbf{v}\\
\exp(\sigma\beta)\mathbf{v}
\end{pmatrix}
\label{eqn:Pzs}
\end{equation}
for some vector $\mathbf{v}$ in $\mathbb{C}^{n}.$
\end{theorem}

The form $\omega^{\beta}$ can be expressed in local coordinates by the matrix
\begin{equation}
\omega^{\beta}=
\begin{pmatrix}
-\beta & \mathbf{1}\\
\mathbf{-1} & 0
\end{pmatrix},
\label{eqn:symp form on zero sec}
\end{equation}
that is $\omega^{\beta}=dx^{j}\wedge dp_{j}-\frac{1}{2}\beta_{jk}dx^{j}\wedge dx^{k}$.

\begin{lemma}
\label{lemma:X_E}The Hamiltonian vector field associated to $E(x,p)=\frac{1}{2}g^{jk}(x)p_{j}p_{k}$ with respect to the twisted symplectic form $\omega^{\beta}=dx^{j}\wedge dp_{j}-\frac{1}{2}\beta_{jk}dx^{j}\wedge dx^{k}$ is given in coordinates as
\[
X_{E}=g^{jk}p_{j}\frac{\partial}{\partial x^{k}}-\left(  \frac{1}{2}
\frac{\partial g^{jk}}{\partial x^{l}}p_{j}p_{k}-\beta_{lj}g^{jk}p_{k}\right)
\frac{\partial}{\partial p_{l}}.
\]
\end{lemma}

\begin{proof}
It is enough to check that $\omega^{\beta}(X_{E},\cdot)=dE$, so we compute 
\begin{align*}
\omega^{\beta}(X_{E}.\cdot)  &  =g^{jk}p_{j}dp_{k}-\frac{1}{2}\beta_{lm} g^{jl}p_{j}dx^{m}+\frac{1}{2}\beta_{jl}g^{kl}p_{k}dx^{j}+\left(  \frac{1}{2}\frac{\partial g^{jk}}{\partial x^{l}}p_{j}p_{k}-\beta_{lj}g^{jk} p_{k}\right)  dx^{l}\\
&  =g^{jk}p_{j}dp_{k}+\frac{1}{2}\frac{\partial g^{jk}}{\partial x^{l}} p_{j}p_{k}\\
&  =dE.
\end{align*}
\end{proof}

\bigskip
\begin{proof}
[Proof of Theorem \ref{thm:zero-section}]Since each point in the zero-section is fixed by the flow, $(\Phi_{\sigma})_{\ast}$ will be a one-parameter group of linear transformations of $T_{(x,0)}T^{\ast}M.$ The infinitesimal generator of this group can be computed, as in \cite[Thm. 3.2]{Hall-K_acs} in the untwisted case, by differentiating $X_{E}.$ We compute
\begin{align*}
\left.  \frac{\partial}{\partial x^{l}}X_{E}\right\vert _{\mathbf{p} =\mathbf{0}}  &  =0\text{, }\\
\left.  \frac{\partial}{\partial p_{l}}X_{E}\right\vert _{\mathbf{p} =\mathbf{0}}  &  =g^{jl}\frac{\partial}{\partial x^{j}}+\beta_{jk}g^{kl} \frac{\partial}{\partial p_{j}}.
\end{align*}
Recalling that the metric tensor at $x$ is the identity, the above is summarized as $
\begin{pmatrix}
\mathbf{0} & \mathbf{1}\\
\mathbf{0} & \beta
\end{pmatrix}.$ Computing the exponential of $\sigma$ times this matrix gives
\[
\left(  \Phi_{\sigma}\right)  _{\ast}=\exp\left[  \sigma
\begin{pmatrix}
\mathbf{0} & \mathbf{1}\\
\mathbf{0} & \beta
\end{pmatrix}
\right]  =
\begin{pmatrix}
1 & \sigma\frac{\exp(\sigma\beta)-\mathbf{1}}{\sigma\beta}\\
0 & \exp(\sigma\beta)
\end{pmatrix}
,
\]
which proves the first part of the theorem. Applying this matrix to a vertical vector, which is represented in coordinates as $
\begin{pmatrix}
0\\
\mathbf{v}
\end{pmatrix}
,$ proves the second part of the theorem.
\end{proof}

\begin{theorem}
\label{thm:analyticity}(Analyticity) Let $\mathcal{L}$ denote the (complex) Lagrangian Grassmannian bundle over $T^{\ast}M$ (with fiber $\mathcal{L}_{z}$). The map $P:T^{\ast}M\times\mathbb{R}\rightarrow\mathcal{L}$ given by (\ref{eqn:mag-geod_flow}) is analytic.
\end{theorem}

\begin{proof}
The argument in the untwisted case that $P$ is analytic \cite[Lemma 3.4]{Hall-K_acs} is in fact valid for any analytic flow. Since by \cite[Prop. 3.7]{Kohno} the magnetic flow is analytic, the theorem follows.
\end{proof}

\bigskip
We want next to show that there exists some $R>0$ such that $\sigma\mapsto P_{z}(\sigma)$ can be analytically continued to a disk of radius greater than one for every $z\in T^{\ast,R}M$ (Theorem \ref{thm:existence}). Suppose $\rho>0$ and our coordinate neighborhood $U$ is such that $(-\rho,\rho)^{n}\subset U.$ By Lemma \ref{lemma:X_E}, the magnetic flow is locally described by the following set of linear differential equations\footnote{In the analytic theory of differential equations, an equation of the form $w^{\prime}=F(w)$, where $F$ is analytic in $w$, is said to be \emph{linear}. This differs of course from the standard theory of ordinary differential equations, where such an equation is said to be linear only when $F$ is linear in $w$.}
\begin{align}
\frac{dx^{l}}{d\sigma}  &  =g^{lj}(\mathbf{x})p_{j},\nonumber\\
\frac{dp_{l}}{d\sigma}  &  =-\frac{1}{2}\frac{\partial g^{jk}(\mathbf{x})}{\partial x^{l}}p_{j}p_{k}+\beta_{lj}(\mathbf{x})g^{jk}(\mathbf{x})p_{k}.
\label{eqn:Ham_flow}
\end{align}
The problem now is to show long-time existence of the solution of (\ref{eqn:Ham_flow}) for points sufficiently close to the zero-section. The reason for long-time existence is that functions on the right hand side of (\ref{eqn:Ham_flow}) are everywhere analytic and zero on the zero-section, and hence for points near the zero-section, the flow is slow moving.

\begin{theorem}
\label{thm:existence}(Existence) There exists $R>0$ such that for some $\varepsilon>0$ we have: (1) for each $z\in T^{\ast,R}M$, the map $\sigma\mapsto P_{z}(\sigma)\in\mathcal{L}_{z}$ given by (\ref{eqn:mag-geod_flow}) admits an analytic continuation to the disk $D_{1+\varepsilon}$ in $\mathbb{C}$ of radius $1+\varepsilon$, and (2) the map $(z,\sigma+i\tau)\mapsto P_{z}(\sigma+i\tau)$ is real analytic as a map of $T^{\ast,R}M\times D_{1+\varepsilon}$ into the Lagrangian Grassmannian bundle over $T^{\ast}M$.
\end{theorem}

Note that since $P_{z}(\sigma)$ is equal to its conjugate for all real values of $\sigma,$ we have $\overline{P_{z}(\sigma)}=P_{z}(\bar{\sigma})$ for all $\sigma\in D_{1+\varepsilon}.$

To prove Theorem \ref{thm:existence}, we need two lemmas which show long-time existence for solutions of $dz/dt=F(z)$ in the neighborhood of a fixed point. Specifically, suppose we are looking for solutions to $dz/dt=F(z)$, where $F$ is analytic with values in an open set $U$ in $\mathbb{C}$. The standard theory of analytic differential equations gives short-time existence of solutions as well as that the solutions depend holomorphically on the initial conditions \cite[Thm. 1.1]{Ilyashenko-Yakovenko}.

\begin{lemma}
\label{lemma:local-short-time-flow}Suppose $F$ is analytic with values in an open set $U$ in $\mathbb{C}$, and suppose a solution $z(t)=\phi(z_{0},t)\in\mathbb{C}^{n}$ to $dz/dt=F(z)$ with initial condition $z(0)=z_{0}$ exists for all $t$ in a disk of radius $R$. If there is a compact set $K\subset U$ such that for all $t$ with $\left\vert t\right\vert <R$, $z(t)$ belongs to $K$, then there exists $S>R$ such that a solution $z(t)$ with $z(0)=z_{0}$ exists for all $t$ with $\left\vert t\right\vert <S$.
\end{lemma}

\begin{proof}
By the method of majorants, solutions exist for times which are locally bounded below. Since $K$ is compact, there is some $\varepsilon>0$ such that for every $w\in K$, a solution in $U$ starting at $w$ exists for all $t\in D_{\varepsilon}$. Consider now a point $t_{0}$ in the boundary of $D_{R}.$ Choose $t_{1}\in D_{R}$ with $\left\vert t_{0}-t_{1}\right\vert <\varepsilon$ and let $z_{1}=\phi(z_{0},t_{1}).$ Then define a map $\tilde{z}:D(t_{1},\varepsilon)\rightarrow U$ by $\tilde{z}(t)=\phi(\phi(z_{0},t_{1}),t-t_{1})$ ($D(t_{1},\varepsilon)$ is the disk of radius $\varepsilon$ centered at $t_{1}$). By uniqueness, $\tilde{z}(t)$ agrees with $\phi(z_{0},t)$ for $t\in D_{R}\cap D(t_{1},\varepsilon).$ Suppose we do this construction for $t_{0}$ and $s_{0}$ on the boundary of $D_{\varepsilon}.$ The intersection of $D(t_{1},\varepsilon)$ with $D(s_{1},\varepsilon)$ is connected and intersects $D_{R}.$ Thus, the $\tilde{z}$ defined on $D(t_{1},\varepsilon)$ and the $\tilde{z}$ defined on $D(s_{1},\varepsilon)$ agree with $\phi(z_{0},t)$ on $D(t_{1},\varepsilon)\cap D(s_{1},\varepsilon)\cap D_{R}$ and hence the two $\tilde{z}$'s agree on $D(t_{1},\varepsilon)\cap D(s_{1},\varepsilon).$ Thus, we get a consistent extension of $\phi(z_{0},t)$ from $D_{R}$ to some open set containing $\overline{D_{R}}$, hence to some $D_{S}$ with $S>R.$
\end{proof}

\begin{lemma}
\label{lemma:local-long-time-flow}Suppose $F(z_{0})=0,$ so that on some ball $B(z_{0},A)\in\mathbb{C}^{n}$ we have
\[
\left\vert F(z)\right\vert \leq C\left\vert z-z_{0}\right\vert .
\]
Then for all $w\in B(z_{0},A)$, a solution starting at $w\neq z_{0}$ exists for all $t$ with
\[
\left\vert t\right\vert <\frac{1}{C}\log\left(  \frac{A}{\left\vert w-z_{0}\right\vert }\right)  .
\]
In particular, as the initial condition approaches the fixed point $z_{0}$, the time-radius of convergence of the associated solution goes to infinity.
\end{lemma}

\begin{proof}
Since $w\neq z_{0}$, the solution $z(t)$ with $z(0)=w$ will never equal $0$, by uniqueness. On any time-disk on which the solution exists,
\[
\frac{d}{dt}\left\vert z(t)-z_{0}\right\vert \leq\left\vert \frac{dz}{dt}\right\vert =\left\vert F(z(t))\right\vert \leq C\left\vert z(t)-z_{0}\right\vert,
\]
which implies
\[
\log\left\vert z(t)-z_{0}\right\vert \leq C\left\vert t\right\vert +\log\left\vert w-z_{0}\right\vert 
\]
whence
\[
\left\vert z(t)-z_{0}\right\vert \leq\left\vert w-z_{0}\right\vert e^{C\left\vert t\right\vert }.
\]
Now, suppose there is some maximal $\tau<\log\left(  A/\left\vert w-z_{0}\right\vert\right)  /C$ such that the solution exists for all $t$ with $\left\vert t\right\vert <\tau$. By the above estimate we have
\[
\left\vert z(t)-z_{0}\right\vert \leq\left\vert w-z_{0}\right\vert e^{C\tau}<A.
\]
By Lemma \ref{lemma:local-short-time-flow}, since the solution stays in the compact set which is the closure of $B(z_{0},\left\vert w-z_{0}\right\vert e^{C\tau})$, the radius $\tau$ was not the maximal radius on which the
solution exists.
\end{proof}

\bigskip
\begin{proof}
[Proof of Theorem \ref{thm:existence}]Given a point $x$ in $M,$ we choose a coordinate system $x^{1},\ldots,x^{n}$ defined in a neighborhood $U$ of $x,$ with the point $x$ corresponding to the origin in the coordinate system. We then use the usual associated coordinate system $x^{1},\ldots,x^{n},p_{1},\ldots,p_{n}$ on $T^{\ast}U\subset T^{\ast}M.$ We can analytically continue the functions on the right-hand side of (\ref{eqn:Ham_flow}) to some connected neighborhood $W$ of the origin in $\mathbb{C}^{2n}.$ Note that the origin in $\mathbb{C}^{2n}$ corresponds to the point $(x,0)\in T^{\ast}M.$

Since the functions on the right-hand side of (\ref{eqn:Ham_flow}) are zero at the origin (since $p=0$ there), Lemma \ref{lemma:local-long-time-flow} tells us that there is a smaller connected neighborhood $W^{\prime}\subset W$ of the origin such for all $z\in W^{\prime},$ a solution in $W$ to the holomorphically extended equations with initial condition $z$ exists for all $\sigma$ in the disk $D_{1+\varepsilon}.$ We denote this solution as $\Phi_{\sigma}(z).$ Applying Lemma \ref{lemma:local-long-time-flow} again, we can find an even smaller connected neighborhood $W^{\prime\prime}\subset W^{\prime}$ of the origin such that for all $z\in W^{\prime\prime}$ and $\sigma\in D_{1+\varepsilon},$ the solution $\Phi_{\sigma}(z)$ belongs to $W^{\prime}$.

For any $z\in W^{\prime\prime}$ and $\sigma\in D_{1+\varepsilon},$ we have $\Phi_{\sigma}(\Phi_{-\sigma}(z))=z.$ After all, the result certainly holds for $z\in W^{\prime\prime}\cap\mathbb{R}^{2n}$ and $\sigma\in D_{1+\varepsilon}\cap\mathbb{R}$. The general result holds because $\Phi_{\sigma}(z)$ depends holomorphically on both $\sigma$ and $z$ and because $W^{\prime\prime}$ and $D_{1+\varepsilon}$ are connected. Since $\Phi_{\sigma}$ has a local inverse, the differential of $\Phi_{\sigma}$ (with respect to the space variables, with $\sigma$ fixed) must be invertible at each point of the form $\Phi_{-\sigma}(z),$ $z\in W^{\prime\prime}.$

Now, at each point in $W,$ we have the ``vertical'' subspace $V_{z}$, namely the directions where the first $n$ components are zero and the last $n$ components are arbitrary. Note that $V_{z}$ is actually independent of $z.$ Now, for $z\in W^{\prime\prime},$ the prescription $P_{z}(\sigma):=(\Phi_{\sigma})_{\ast}(V_{\Phi_{-\sigma}(z)})$ makes sense, and because $(\Phi_{\sigma})_{\ast}$ is invertible, $P_{z}(\sigma)$ is an $n$-dimensional complex subspace of $\mathbb{C}^{2n}.$ Furthermore, $P_{z}(\sigma)$ depends holomorphically on $\sigma$ for each $z$---as a map into the Grassmannian of $n$-dimensional complex subspaces of $\mathbb{C}^{2n}$---because $\Phi_{\sigma}$ and $\Phi_{-\sigma}$ depend holomorphically on $\sigma$ and because $V_{z}$ is independent of $z.$ In particular, this is true for $z\in W^{\prime\prime}\cap\mathbb{R}^{2n},$ which corresponds to a neighborhood of $(x,0)$ in $T^{\ast}M.$ Thus, there is a neighborhood of $(x,0)$ on which the map $\sigma\mapsto P_{z}(\sigma)$ can be analytically continued to $D_{1+\varepsilon}.$ Since this is true for each $x\in M,$ compactness shows that the desired analytic continuation exists for all $z$ in some $T^{\ast,R}M.$

From the formula for $P_{z}(\sigma),$ it should clear that $P_{z}(\sigma)$ depends analytically on $z$ and $\sigma.$ Finally, since $P_{z}(\sigma)$ is Lagrangian for all $\sigma\in\mathbb{R}$ and the Grassmannian of all complex Lagrangian subspaces of $T_{z}^{\mathbb{C}}T^{\ast}M$ is a complex submanifold of the Grassmannian of all $n$-dimensional subspaces of $T_{z}^{\mathbb{C}}T^{\ast}M,$ we see that $P_{z}(\sigma)$ is Lagrangian for all $\sigma\in D_{1+\varepsilon}.$
\end{proof}

\begin{remark}
It follows from the proof of Theorem \ref{thm:existence} that we can in fact ensure existence for arbitrarily long times by considering points sufficiently close to the zero-section, that is, that for every $T>0$ there exists $R>0$ such that for each $z\in T^{\ast,R}M,$ the map $\sigma\mapsto P_{z}(\sigma)$ can be continued to $D_{T}.$
\end{remark}

\begin{theorem}
\label{thm:integrability}(Integrability) Suppose $R$ is such that there exists $\varepsilon>0$ such that: (1) for each $z\in T^{\ast,R}M$, the map $\sigma\mapsto P_{z}(\sigma)\in\mathcal{L}_{z}$ admits a holomorphic extension to a disk $D_{1+\varepsilon},$and (2) the map $(z,\sigma+i\tau)\mapsto P_{z}(\sigma+i\tau)$ is smooth as a map of $T^{\ast,R}M\times D_{1+\varepsilon}$ into the Lagrangian Grassmannian bundle over $T^{\ast}M$. Then the bundle
$P_{z}(\sigma+i\tau)$ is integrable.
\end{theorem}

\begin{proof}
The proof is the same as in the untwisted case \cite[Thm 3.6]{Hall-K_acs}, which relies only on the analyticity of the flow and the integrability of the vertical distribution.
\end{proof}

\begin{theorem}
\label{thm:Kaehler}(Complex structure and K\"{a}hler structure) Let $\varepsilon$ and $R$ be as in Theorem \ref{thm:existence}. Then there is some $r\leq R$ such that the following hold.

\begin{enumerate}
\item For all $z\in T^{\ast,r}M$ and all $\sigma+i\tau\in D_{1+\varepsilon}$ with $\tau\neq0,$ the subspace $P_{z}(\sigma+i\tau)$ intersects its conjugate only at zero. On any such $T^{\ast,r}M,$ there is a unique complex structure whose $(1,0)$-tangent distribution is $P_{z}(\sigma+i\tau).$

\item For all $z\in T^{\ast,r}M$ and all $\sigma+i\tau\in D_{1+\varepsilon}$ with $\tau>0,$ the subspace $P_{z}(\sigma+i\tau)$ is a positive Lagrangian subspace of $T_{z}^{\mathbb{C}}T^{\ast}M$ with respect to $\omega^{\beta}.$
\end{enumerate}
\end{theorem}

Whenever there is a (unique) complex structure on some $T^{\ast,R}M$ whose $(1,0)$-distribution is $P_{z}(i),$ we refer to that structure as the \textbf{magnetic complex structure} on $T^{\ast,R}M$ (relative to the fixed metric $g$ and magnetic field $\beta$ on $M$).

\bigskip
\begin{proof}
First, we show that on the zero-section, $P_{(x,0)}(\sigma+i\tau)$ intersects its conjugate only at zero. The intersection of a space with its conjugate is invariant under conjugation. If the intersection contains a nonzero vector $Y,$ then at least one of $X=(Y+\bar{Y})/2$ and $X=(Y-\bar{Y})/(2i)$ would be nonzero and satisfy $X=\bar{X}.$ So it suffices to show that any $X\in P_{(x,0)}(\sigma+i\tau)$ with $X=\bar{X}$ must be zero. Equating the first $n$ components of a vector of the form (\ref{eqn:Pzs}) to the first $n$ components of its conjugate gives
\begin{equation}
-2i\tau e^{\sigma\beta}\frac{\sin(\tau\beta)}{\tau\beta}\mathbf{v}=0.
\label{xXbar}
\end{equation}
Since $\beta$ is skew-symmetric, the eigenvalues of $\tilde{\beta}$ are pure imaginary. Thus, the eigenvalues of $\sin(\tau\beta)/(\tau\beta)$ are of the form $\sinh(\tau\lambda)/(\tau\lambda),$ for $\lambda\in\mathbb{R},$ showing that $\sin(\tau\beta)/(\tau\beta)$ is nonsingular. Since $e^{\sigma\beta}$ is also nonsingular, the only solution to (\ref{xXbar}) is $\mathbf{v}=0,$ for any $\tau\neq0.$ Once we know that $P_{z}(\sigma+i\tau)$ is disjoint from its conjugate for $z$ in the zero-section, a standard compactness argument shows that $P_{z}(\sigma+i\tau)$ continues to be disjoint from its conjugate on some tube $T^{\ast,R}M.$

To see that $P_{(x,0)}(\sigma+i\tau)$ is positive, we will show that for each $Z\in P_{(x,0)}(\sigma+i\tau)$ we have
\begin{equation}
-i\omega^{\beta}(Z,\overline{Z})>0. \label{eqn:pos cond}
\end{equation}
Using equation (\ref{eqn:symp form on zero sec}), Theorem \ref{thm:zero-section}, and the skew symmetry of $\beta,$ we compute that
\begin{equation}
-i\omega^{\beta}\left(
\begin{pmatrix}
\frac{\exp((\sigma+i\tau)\beta-\mathbf{1})}{\beta}\mathbf{v}\\
\exp((\sigma+i\tau)\beta)\mathbf{v}
\end{pmatrix}
,
\begin{pmatrix}
\frac{\exp((\sigma-i\tau)\beta)-\mathbf{1}}{\beta}\mathbf{v}\\
\exp((\sigma-i\tau)\beta)\mathbf{v}
\end{pmatrix}
\right)  =2\tau
\begin{pmatrix}
\mathbf{v}\cdot\frac{\mathbf{1}-\exp(-2i\tau\beta)}{2i\tau\beta}\mathbf{v}
\end{pmatrix}
. \label{positivityCalc}
\end{equation}
Since $\beta$ has imaginary eigenvalues, $2i\tau\beta$ has real eigenvalues. Also, $(1-e^{-x})/x$ is positive for all real values of $x$, the matrix on the right-hand side of (\ref{positivityCalc}) is positive definite. Thus for all $\tau>0,$ the condition (\ref{eqn:pos cond}) holds for all nonzero $Z\in P_{(x,0)}(\sigma+i\tau)$. Again, once the desired positivity is established on the zero-section, a simple compactness argument shows that it holds on some $T^{\ast,R}M$ (not necessarily the same $R$ as in Theorem \ref{thm:existence}). This positivity, combined with the fact that the distribution $P(\sigma+i\tau)$ is Lagrangian, shows that the magnetic complex structure is K\"{a}hler on some $T^{\ast,R}M.$
\end{proof}

\bigskip
Finally, we conclude this section with a proof that the zero-section is totally real with respect to the magnetic complex structure.

\begin{theorem}
\label{thm:tot-real}The zero-section in $T^{\ast,R}M$ is a maximal totally real submanifold with respect to the magnetic complex structure.
\end{theorem}

\begin{proof}
If the zero-section were \textit{not} totally real, then there would exist a nonzero vector of type $(1,0)$ in the complexified tangent space to the zero-section at some point $(x,0)$. Now, the complexified tangent space to the zero-section is represented in coordinates by $2n$-component complex vectors whose last $n$ components are zero, and the $(1,0)$ subspace at $(x,0)$ is obtained by putting $\sigma=i$ in (\ref{eqn:Pzs}). Since $\exp(i\beta)$ is invertible, the last $n$ components of the vector are $0$ only if $\mathbf{v}=0,$ in which case the whole vector is zero.
\end{proof}

\section{An approach using Bruhat--Whitney embeddings\label{grauert.sec}}

In this section, we give a construction of the magnetic complex structure in terms of a Bruhat--Whitney embedding of $M$ into a complex manifold $X.$ We continue to assume that $(M,g)$ is a compact, real-analytic Riemannian manifold and that $\beta$ is a closed, real-analytic 2-form on $M.$

\begin{theorem}
\label{thm:piPhi}Suppose $M$ is real analytically embedded as a totally real submanifold of maximal dimension in a complex manifold $X.$ (Such an embedding exists by \cite{Bruhat-Whitney}, \cite{Grauert58}.) Then there exist positive numbers $R$ and $\varepsilon$ such that $z\in T^{\ast,R}M$, the map
\[
\sigma\in\mathbb{R}\mapsto\pi\circ\Phi_{\sigma}(z)\in M\subset X
\]
can be analytically continued to $D_{1+\varepsilon}$ and such that the\ resulting map $\pi\circ\Phi_{i}$ is a diffeomorphism of $T^{\ast,R}M$ onto its image in $X$.
\end{theorem}

\begin{proof}
We can deduce the first part of the theorem from the proof of Theorem \ref{thm:existence} as follows. We fix some $x_{0}\in M$ and we choose real-analytic coordinates $x^{1},\ldots,x^{n}$ in a neighborhood of $x_{0},$ with the origin of the coordinate system corresponding to $x_{0}.$ We may then analytically continue the $x^{j}$'s to holomorphic coordinates (still called $x^{j}$) defined in a neighborhood $U$ of $x_{0}$ in $X.$ We use the usual cotangent bundle coordinates $\{x^{j},p_{j}\}$ on $T^{\ast}M$ and we analytically continue all the functions occurring on the right-hand side of (\ref{eqn:Ham_flow}) to some neighborhood $V$ of the origin in $\mathbb{C}^{2n}.$ By shrinking $V$ if necessary, we can assume that the projection map $\pi(x,p):=x$ on $V$ maps into $U.$ By the argument in the proof of Theorem \ref{thm:existence}, we can choose a neighborhood $W\subset V$ of the origin in $\mathbb{C}^{2n}$ such that for all $(x,p)\in W,$ the solution to the analytically continued version of (\ref{eqn:Ham_flow}) exists in $V$ for all $\sigma\in D_{1+\varepsilon}$ and depends holomorphically on $\sigma.$ Thus, $\pi\circ\Phi_{\sigma}$ is defined and holomorphic in $\sigma$ for all $(x,p)\in W,$ as a map into $U\subset X.$ This establishes the desired analytic continuation in some neighborhood of each point $(x,0)$ in the zero-section, and thus, by compactness on some $T^{\ast,R}M.$

To see that the map $\pi\circ\Phi_{i}$, which is well defined by the previous paragraph, is a diffeomorphism we will show that the differential of the map $\pi\circ\Phi_{i}$ is nonsingular on the zero-section, and thus also on a neighborhood of the zero-section, and thus on $T^{\ast,r}M$ for some $r\leq R.$ We now argue that $\pi\circ\Phi_{i}$ is injective on some $T^{\ast,r^{\prime}}M,$ for $r^{\prime}\leq r.$ If this were not the case, then there would exist sequences $(x^{(j)},p^{(j)})$ and $(\tilde{x}^{(j)},\tilde{p}^{(j)})$ with $\left\vert p^{(j)}\right\vert $ and $\left\vert \tilde{p}^{(j)}\right\vert $ tending to zero such that for each $j,$ we have $(x^{(j)},p^{(j)})\neq(\tilde{x}^{(j)},\tilde{p}^{(j)})$ but $\pi\circ\Phi_{i}(x^{(j)},p^{(j)})=\pi\circ\Phi_{i}(\tilde{x}^{(j)},\tilde{p}^{(j)}).$ By passing to subsequences, we can assume that $x^{(j)}$ and $\tilde{x}^{(j)}$ converge to $x$ and $\tilde{x},$ respectively, in which case the continuity of $\pi\circ\Phi_{i}$ would tell us that
\[
x=\pi\circ\Phi_{i}(x,0)=\pi\circ\Phi_{i}(\tilde{x},0)=\tilde{x}.
\]
But then $(x^{(j)},p^{(j)})$ and $(\tilde{x}^{(j)},\tilde{p}^{(j)})$ would be converging to the same point in the zero-section, which would mean that $\pi\circ\Phi_{i}$ fails to be locally injective near $(x,0),$ which would contradict the Inverse Function Theorem.

By Theorem \ref{thm:zero-section} and the trivial fact that $\pi_{\ast}=\begin{pmatrix}
\mathbf{1} & 0
\end{pmatrix}
$, we see that at $(x,0)$ in the zero-section of $T^{\ast}M$, the differential of $\pi\circ\Phi_{\sigma}$ is
\begin{equation}
(\pi\circ\Phi_{\sigma})_{\ast}=
\begin{pmatrix}
\mathbf{1} & \frac{\exp(\sigma\beta)-\mathbf{1}}{\beta}
\end{pmatrix}
. \label{piPhiDiff}
\end{equation}
Now, we may identify the complexified tangent space $(T_{x}M)_{\mathbb{C}}$ with $T_{x}X$ by mapping $iX$ to $JX.$ Since $\pi\circ\Phi_{\sigma}$ varies holomorphically in $X$ with respect to $\sigma$ at each point, the differential $(\pi\circ\Phi_{\sigma})_{\ast}(x,0)$ will vary holomorphically in $T_{x}X$ with respect to $\sigma.$ Thus, (\ref{piPhiDiff}) holds for all $\sigma\in D_{1+\varepsilon}.$ Thus, if we represent the differential using the basis $\{\partial/\partial x^{j},\partial/\partial p_{j}\}$ on the domain side and the basis $\{\partial/\partial x^{j},J(\partial/\partial x^{j})\}$ on the range side, we obtain the matrix
\[
=\left(
\begin{array}
[c]{cc}
\mathbf{1} & \frac{\cos\beta-\mathbf{1}}{\beta}\\
0 & \frac{\sin\beta}{\beta}
\end{array}
\right)  .
\]
The determinant of this matrix is $\det[\sin\beta/\beta],$ which is nonzero because the eigenvalues of $\beta$ are pure imaginary. Thus, the differential of $\pi\circ\Phi_{i}$ is nondegenerate at $(x,0)$, as claimed.
\end{proof}

\begin{theorem}
\label{thm:Xmcs}Let $R_{1}$ be as in Theorem \ref{thm:piPhi}, let $R_{2}$ be such that the magnetic complex structure is defined on $T^{\ast,R_{2}}M,$ and let $R=\min(R_{1},R_{2}).$ Then the pullback of the canonical complex structure on $X$ by $\pi\circ\Phi_{i}$ is the magnetic complex structure on $T^{\ast,R}M$; that is, if $T^{\ast,R}M$ is equipped with the magnetic complex structure, then $\pi\circ\Phi_{i}$ is holomorphic.
\end{theorem}

\begin{proof}
For any real value of $\sigma,$ each vector $X$ in $P_{z}(-\sigma)$ is, by definition, of the form $(\Phi_{-\sigma})_{\ast}(Y)$ for some vertical vector $Y$ at the point $\Phi_{\sigma}(z).$ Thus,
\[
(\pi\circ\Phi_{\sigma})_{\ast}(X)=(\pi\circ\Phi_{\sigma})_{\ast} (\Phi_{\sigma})_{\ast}(Y)=(\pi\circ\Phi_{\sigma}\circ\Phi_{-\sigma})_{\ast}(Y)=\pi_{\ast}(Y)=0.
\]
That is to say, $(\pi\circ\Phi_{\sigma})_{\ast}(z)$ is zero on $P_{z}(-\sigma),$ for all $\sigma\in\mathbb{R}$ and $z\in T^{\ast}M.$ Now, we have already noted that $(\pi\circ\Phi_{\sigma})_{\ast}(z)$ depends holomorphically on $\sigma$ for each fixed $z\in T^{\ast,R}M.$ Furthermore, $P_{z}(-\sigma)$ depends holomorphically on $-\sigma$ and hence also on $\sigma.$ Using local holomorphic frames for $P_{z}(\sigma)$ as in the proof of Theorem 8 in \cite{Hall-K_acs}, it follows that $(\pi\circ\Phi_{\sigma})_{\ast}(z)$ is zero on $P_{z}(-\sigma)$ for all $z\in T^{\ast,R}M$ and $\sigma\in D_{1+\varepsilon}.$ In particular, $(\pi\circ\Phi_{i})_{\ast}(z)$ is zero on $P_{z}(-i)=\overline{P_{z}(i)}.$ This shows that $\pi\circ\Phi_{i}$ is holomorphic. 
\end{proof}

\section{Holomorphic functions\label{functions.sec}}

In this section, we first describe the magnetic complex structure in terms of holomorphic functions. As in the untwisted case, functions on $T^{*,r}M$, for $0<r\leq R$, which are holomorphic with respect to the magnetic complex structure associated to $\beta$ can be expressed as the composition of some vertically constant function on $T^{\ast}M$ with the ``time-$i$'' magnetic flow \cite[Thm. 14]{Hall-K_acs}. Although the proof given in \cite{Hall-K_acs} does not  apply in the untwisted case,  we give here a simple proof, which works also in the untwisted case. We continue to assume that $(M,g)$ is a compact, real-analytic Riemannian manifold and that $\beta$ is a closed, real-analytic 2-form on $M.$

\begin{theorem}
Suppose that for some $R>0$, $T^{\ast,R}M$ admits a magnetic complex structure. Then the following hold.

\begin{enumerate}
\item \label{item:func1}Suppose $f:M\rightarrow\mathbb{C}$ is a real-analytic function. Then there exists $r\in(0,R]$ such that for all $z\in T^{\ast,r}M$, the function
\[
\sigma\mapsto f\circ\pi\circ\Phi_{\sigma}(z)
\]
can be analytically continued to a disk of radius greater than $1$, and, when $\sigma=i$, the resulting function
\[
f\circ\pi\circ\Phi_{i}:T^{\ast,r}M\rightarrow\mathbb{C}
\]
is holomorphic with respect to the magnetic complex structure.

\item \label{item:func2}Suppose $F$ is a holomorphic function with respect to the magnetic complex structure on some $T^{\ast,r}M$ with $r\in(0,R].$ Then there exists $s\in(0,r]$ such that on $T^{\ast,s}M,$
\begin{equation}
F=f\circ\pi\circ\Phi_{i}, \label{eqn:f_C expr}
\end{equation}
where $f$ is the restriction of $F$ to the zero-section, and (\ref{eqn:f_C expr}) is interpreted as in Point 1.
\end{enumerate}
\end{theorem}

\begin{proof}
(\ref{item:func1}) We embed $M$ into $X$ as in Theorem \ref{thm:Xmcs}. Since $f$ is real-analytic, it can be analytically continued to some neighborhood $U$ of $M$ in $X.$ Applying Theorem \ref{thm:piPhi} to the complex manifold $U,$ we can find positive numbers $R$ and $\varepsilon$ such that for all $z\in T^{\ast,R}M$, the map $\sigma\mapsto\pi\circ\Phi_{\sigma}(z)$ can be analytically continued to a map of $D_{1+\varepsilon}$ into $U$. Then for each $z\in T^{\ast,R}M,$ the map $\sigma\mapsto f(\pi\circ\Phi_{\sigma})(z),$ is holomorphic on $D_{1+\varepsilon}$. Furthermore, since both $f:U\rightarrow\mathbb{C}$ and $\pi\circ\Phi_{i}$ are holomorphic, the composition $f\circ\pi\circ\Phi_{i}:T^{\ast,r}M\rightarrow\mathbb{C}$ is holomorphic.

(\ref{item:func2}) By Part (\ref{item:func1}), the restriction $f$ of $F$ to the zero-section can be analytically continued to a holomorphic function $\tilde{F}:T^{\ast,s}M\rightarrow\mathbb{C}$ for some $s\in(0,R].$ Since the zero-section is a maximal totally real submanifold and $F$ agrees with $\tilde{F}$ on the zero-section, $F$ must agree with $F$ everywhere on $T^{\ast,t}M,$ where $t=\min(s,r)\leq r.$
\end{proof}

\begin{remark}
In the untwisted case, one may compute that for fixed $x$, the function
\[
\left(X_{E}^{\beta=0}\right)^{k}(f\circ\pi)(x,p)
\]
is a homogeneous polynomial of degree $k$ in $p.$ Thus, one obtains an expression for $f\circ\pi\circ\Phi_{i}^{\beta=0}$ which can be described as a ``Taylor series in the fiber'' \cite[Eqn.
(15)]{Hall-K_acs}. In the twisted $\beta\neq0$ case, for fixed $x$ the quantity $\left(  X_{E}^{{}}\right)  ^{k}(f\circ\pi)(x,p)$ is no longer a homogeneous polynomial of degree $k$ in $p,$ as may be seen already in the $\mathbb{R}^{2}$ case (Section \ref{sec:R2case}).
\end{remark}

\section{K\"{a}hler potentials and holomorphic sections\label{sec:K}}

We continue to assume that $(M,g)$ is a real-analytic Riemannian manifold and that $\beta$ is a closed, real-analytic 2-form on $M.$ In this section, we consider the problems of finding a local K\"{a}hler potential and finding local holomorphic sections of line bundles over $T^{\ast,R}M.$ Let $\theta$ denote the canonical 1-form on $T^{\ast}M,$ given in coordinates as $\theta=p_{j}dx^{j},$ and define a 1-form $\theta^{A}$ by
\begin{equation}
\theta^{A}:=\theta+\pi^{\ast}A. \label{thetaAdef}
\end{equation}
Then $\theta^{A}$ is a symplectic potential for $\omega^{\beta}$; that is, $-d\theta^{A}=\omega-\pi^{\ast}\beta=\omega^{\beta}.$ We give sufficient conditions (Corollary \ref{thm:suffK}) for the existence of a local K\"{a}hler potential which is adapted to the geometry of the twisted tangent bundle, in the sense that the imaginary part of the antiholomorphic derivative of the potential is equal to $\theta^A$. We also give a general formula for a K\"{a}hler potential, which is not in general adapted to $\theta^{A}.$

Before we state and prove our results, we briefly remind the reader of the connection between trivializations of holomorphic line bundles and K\"{a}hler potentials. Let $(X,\omega,J)$ be a K\"{a}hler manifold with integral K\"{a}hler form $\omega$ and suppose that $L\rightarrow X$ is a prequantum line bundle\footnote{A line bundle $L\rightarrow(X,\omega)$ is a \emph{prequantum} line bundle if it is hermitian line bundle with compatible connection $\nabla^{L}$ with curvature $curv\nabla^{L}=-i\omega$. Such a bundle exists if and only if the class $[\omega/2\pi]\in H_{dR}^{2}(X,\mathbb{Z})$ is integral.} with holomorphic structure induced by $J.$ Suppose $U$ is a simply connected open set in $X$ such that $L$ is trivial over $U,$ and suppose $\theta$ is a symplectic potential, that is, a 1-form on $U$ satisfying $-d\theta=\omega.$ Then it is possible to choose a unitary trivialization of the $k$th tensor power $L^{\otimes k}$ of $L$ over $U$ in such a way that the connection $\nabla^{L}$ acts on sections (which are now identified with functions) as
\begin{equation}
\nabla_{X}^{L}f=Xf+ik\theta(X)f. \label{localConnection}
\end{equation}
A section $e^{-k\kappa/2}$ of $L^{\otimes k}$ over $U$ will thus be holomorphic if and only if
\[
0=\nabla_{X}^{L}\left(  e^{-k\kappa/2}\right)  =-\frac{k}{2}(X\kappa)e^{-\kappa/2}+ik\theta(X)e^{-\kappa/2}
\]
for every vector $X$ of type $(0,1),$ which holds if and only if
\begin{equation}
\bar{\partial}\kappa=2i\theta^{(0,1)}. \label{eqn:local-hol-cond}
\end{equation}

Now, if a function $\kappa$ satisfying (\ref{eqn:local-hol-cond}) turns out to be real valued, then
\begin{equation}
\operatorname{Im}\bar{\partial}\kappa=\theta^{(0,1)}+\theta^{(1,0)}=\theta\label{eqn:adaptedKcond}
\end{equation}
which implies that
\begin{equation}
i\partial\bar{\partial}\kappa=\omega, \label{eqn:Kpot-def}
\end{equation}
which says that $\kappa$ is a K\"{a}hler potential.

We now specialize to the problem at hand, where our 2-form is $\omega^{\beta}:=\omega-\pi^{\ast}\beta$ and where our symplectic potential is $\theta^{A}:=\theta+\pi^{\ast}A,$ with $A$ satisfying $dA=\beta$ locally. We then wish to solve (\ref{eqn:local-hol-cond}) with $\theta$ replaced by $\theta^{A}.$ In Theorem \ref{thm:hol_sec}, we obtain a solution $\kappa=2if_{-i},$ but this function is not in general real valued. Nevertheless, $\exp(-if_{-i})$ is a local holomorphic section, which is all that one may care about in certain applications. Still, if we want to obtain a (real-valued) K\"{a}hler potential, there are two approaches we can take. First, since the operator $i\partial\bar{\partial}$ maps real-valued functions to real-valued 2-forms, the real part of $2if_{-i},$ which may be computed as $i(f_{-i}-f_{i}),$ is a K\"{a}hler potential. This function, however, may not be adapted to the symplectic potential $\theta^{A}.$ That is, the function $\kappa:=i(f_{i}-f_{-i})$ satisfies $i\partial\bar{\partial}\kappa =\omega^{\beta},$ but may not satisfy $\operatorname{Im}\bar{\partial}\kappa=\theta^{A}.$ Second, if there exists a holomorphic function $g$ such that $2if_{-i}+g$ is real valued, then this function still satisfies (\ref{eqn:local-hol-cond}), since $\bar{\partial}g=0$, and is now real valued. In general, however, there does not appear to be any reason such a $g$ must exist.

\begin{theorem}
\label{thm:hol_sec}Suppose $R>0$ is such that the magnetic complex structure exists on $T^{\ast,R}M.$ Given $x\in M,$ let $A$ be a real-analytic magnetic potential defined in a neighborhood $U$ of $x.$ Then there exist an $\varepsilon>0$ and a neighborhood $V$ of $(x,0)$ in $T^{\ast}M$ with $\pi(V)\subset U$ such that the following hold.
\begin{enumerate}
\item \label{item:hol_sec1}For all $z\in V,$ the function
\[
\sigma\mapsto\int_{0}^{\sigma}A\left(  \frac{d(\pi\circ\Phi_{s})(z)}
{ds}\right)  ds,
\]
can be analytically continued to a disk of radius $1+\varepsilon.$

\item \label{item:hol_sec2}Let
\[
f_{\sigma}(z):=\sigma E(z)+\int_{-\sigma}^{0}A\left(  \frac{d(\pi\circ\Phi_{s})(z)}{ds}\right)  ds,\quad z\in V^{R}(\varepsilon),
\]
which by Point \ref{item:hol_sec1} may be defined for all $\sigma\in D_{1+\varepsilon}$. Then the function $f_{-i}$ satisfies
\[
\bar{\partial}f_{-i}=(\theta^{A})^{(0,1)}.
\]
\end{enumerate}
\end{theorem}

Note that since $\pi(V)\subset U,$ the magnetic potential is defined on points of the form $\pi\circ\Phi_{s}$ for all sufficiently small $s,$ so that it makes sense to speak about the analytic continuation of the function in Point \ref{item:hol_sec1}. Point \ref{item:hol_sec2} tells us that the function $2if_{-i}$ has all the properties necessary to be a K\"{a}hler potential \textit{except} that this function may not be real valued. Corollaries \ref{thm:suffK} and \ref{thm:K} to Theorem \ref{thm:hol_sec} give two different approaches to finding a (real-valued) K\"{a}hler potential.

\begin{corollary}
\label{cor:localSection}Let $L\rightarrow T^{\ast,R}M$ be a prequantum line bundle with complex structure determined by the magnetic complex structure. Then in the local trivialization of $L^{\otimes k}$ determined by $\theta^{A}$, the function $\exp\{-ikf_{-i}(z)\}$ is a holomorphic section.
\end{corollary}

\begin{corollary}
\label{thm:K}The function
\[
\operatorname{Re}(2if_{-i})=i(f_{-i}-f_{i})
\]
is a K\"{a}hler potential.
\end{corollary}

Since $f_{\sigma}$ is real valued when $\sigma$ is real, we have $\overline{f_{\sigma}}=f_{\bar{\sigma}},$ from which the formula for $\operatorname{Re}(2if_{-i})$ follows. It should be noted that the K\"{a}hler potential in Corollary \ref{thm:K} is not necessarily adapted to the symplectic potential $\theta^{A}$; that is $\operatorname{Im}\bar{\partial}[i(f_{-i}-f_{i})]$ is not necessarily equal to $\theta^{A},$ as the $\mathbb{R}^{2}$ example shows (Section \ref{sec:R2case}).

\begin{corollary}
\label{thm:suffK}Let $f_{\sigma}(z)$ be as in Theorem \ref{thm:hol_sec}. If there exists a holomorphic function $g$ such that $2if_{-i}+g$ is real valued, then $2if_{-i}+g$ is a K\"{a}hler potential adapted to the symplectic potential $\theta^{A}$ in (\ref{thetaAdef}). That is,
\[
\operatorname{Im}\bar{\partial}(2if_{-i}+g)=\theta^{A}.
\]

\end{corollary}

We turn now to the proof of Theorem \ref{thm:hol_sec}, which we break up into the following lemmas. The first is simply Theorem \ref{thm:hol_sec}(\ref{item:hol_sec1}).

\begin{lemma}
For all $z$ in some neighborhood of $(x,0)$, the function $\sigma\mapsto \int_{0}^{\sigma}A\left(  \frac{d(\pi\circ\Phi_{s})(z)}{ds}\right)  ds$ can be continued to a disk of radius greater than $1$.
\end{lemma}

\begin{proof}
As in the proof of Theorem \ref{thm:existence}, we can make sense of $\Phi_{s}$ for complex values of $s,$ by working in local coordinates and holomorphically extending the functions on the right-hand side of (\ref{eqn:Ham_flow}). Of course, the map $\pi$ also analytically continues to the map sending $(x,p)$ to $x$ for all $(x,p)\in\mathbb{C}^{2n}.$ Finally, the components of $A=A_{j}dx^{j}$ are assumed to be real analytic, so they also analytically continue to some neighborhood of the origin in $\mathbb{C}^{2n}.$ Furthermore, the proof shows that if we start out sufficiently close to the origin, then we can ensure that $\Phi_{s}$ stays in any specified neighborhood of the origin for time $1+\varepsilon.$ Thus, we can find some neighborhood on which $\int_{0}^{\sigma}A\left(  \frac {d(\pi\circ\Phi_{s})(z)}{ds}\right)  ds$ makes sense for $\sigma\in D_{1+\varepsilon}$. If we rewrite the integral as $\sigma\int_{0}^{1}A\left(\frac{d(\pi\circ\Phi_{\sigma s})(z)}{ds}\right)  ds,$ then because $\Phi_{\sigma}(z)$ depends holomorphically on $\sigma$ and $z,$ we see that the integral depends holomorphically on $\sigma.$
\end{proof}

Part (\ref{item:hol_sec2}) of Theorem \ref{thm:hol_sec} comes from the following three lemmas. For a vector field $X$, we define
\begin{equation}
X^{\sigma}:=(\Phi_{\sigma})_{\ast}X. \label{eqn:vf_geodPF}
\end{equation}

\begin{lemma}
\label{lemma:thetaA-Xsig_formula}Let $R,\varepsilon$ be as in Theorem \ref{thm:existence} and let $f_{\sigma}$ be a family of functions defined on an open set $V\subset T^{\ast,R}M,$ where $\sigma$ varies over an open interval $(-a,a)\subset\mathbb{R}.$ Suppose that for all $\sigma\in(-a,a),$ we have
\begin{equation}
\theta^{A}(X)=df_{\sigma}(X) \label{eqn:preK}
\end{equation}
for each $X\in P(\sigma).$ Suppose also that for all $z\in V$, the function $\sigma\mapsto f_{\sigma}(z)$ can be analytically continued to $D_{1+\varepsilon}.$ Then (\ref{eqn:preK}) continues to hold for all $\sigma\in D_{1+\varepsilon}.$
\end{lemma}

\begin{proof}
If $f_{\sigma}$ is defined by analytic continuation at each point, then at each fixed point $z,$ the linear functional $df_{\sigma}(z)$ depends holomorphically on $\sigma$. Furthermore, the subspaces $P_{z}(\sigma)$ depend holomorphically on $\sigma,$ by construction. Thus, using locally defined holomorphic frames for the family $P_{z}(\sigma),~\sigma\in D_{1+\varepsilon},$ as in the proof of Theorem 8 in \cite{Hall-K_acs}, it is straightforward to show that if (\ref{eqn:preK}) holds for $\sigma\in(-a,a),$ it must continue to hold for all $\sigma\in D_{1+\varepsilon}.$
\end{proof}

By the previous lemma, Theorem \ref{thm:hol_sec}(\ref{item:hol_sec2}) is now reduced to finding a function $f_{\sigma}(z)$ which satisfies (\ref{eqn:preK}). The next lemma gives a sufficient condition under which (\ref{eqn:preK}) is satisfied.

\begin{lemma}
\label{kde.lem}Suppose a function $f_{\cdot}(\cdot):(-a,a)\times V\rightarrow\mathbb{R}$ has the following properties: (1) $f$ satisfies the differential equation
\begin{equation}
\frac{\partial f_{\sigma}}{\partial\sigma}=-X_{E}(f_{\sigma})+(\theta
^{A}(X_{E})-E) \label{eqn:KDE}
\end{equation}
and (2) for every vertical vector $X,$ we have $X(f_{0})=0$. Then $f$ satisfies (\ref{eqn:preK}).
\end{lemma}

\begin{proof}
Let $X$ be a vertical vector field,let $X^{\sigma}$ be defined by (\ref{eqn:vf_geodPF}), and let 
\[
u(\sigma)=\theta^{A}(X^{\sigma})-X^{\sigma}(f_{\sigma}).
\]
Since $\partial X^{\sigma}/\partial\sigma=[X^{\sigma},X_{E}]$, we have
\[
\frac{\partial}{\partial\sigma}u(\sigma)=\theta^{A}([X^{\sigma},X_{E}
])-[X^{\sigma},X_{E}](f_{\sigma})-X^{\sigma}\left(  \frac{\partial f_{\sigma}
}{\partial\sigma}\right)  .
\]
Now, by a standard identity for exterior derivatives, we have
\[
\theta^{A}([X^{\sigma},X_{E}])=X^{\sigma}\theta^{A}(X_{E})-X_{E}\theta^{A}(X^{\sigma})-d\theta^{A}(X^{\sigma},X_{E}),
\]
which can be simplified by noting that
\[
-d\theta^{A}(X^{\sigma},X_{E})=\omega^{\beta}(X_{\sigma},X_{E})=dE(X^{\sigma})=-X^{\sigma}E.
\]
Thus,
\begin{align*}
\frac{\partial}{\partial\sigma}u(\sigma)  &  =-X_{E}\left[  \theta^{A}(X^{\sigma})-X^{\sigma}(f_{\sigma})\right]  -X^{\sigma}\left[\frac{\partial f}{\partial\sigma}+X_{E}(f_{\sigma})-\left(  \theta^{A}(X_{E})-E\right)  \right] \\
&  =-X_{E}(u(\sigma)),
\end{align*}
where we have used the differential equation in the second equality.

Now, the unique solution to $\partial u(\sigma)/\partial\sigma=-X_{E}(u(\sigma))$ is simply $u(\sigma)(z)=u(0)(\Phi_{-\sigma}(z)).$ Since $\theta^{A}$ is zero on vertical vectors and $X(f_{0})=0,$ we see that $u(0)=0$ whence $u(\sigma)=0$ for all $\sigma$.
\end{proof}

Finally, we show that the function $f_{\sigma}$ defined in Theorem \ref{thm:hol_sec} satisfies the sufficient condition of the previous lemma, thus completing the proof of Theorem \ref{thm:hol_sec}.

\begin{lemma}
\label{lemma:fsig}For each~$x\in M,$ the function
\begin{equation}
f_{\sigma}(z):=\sigma E(z)\sigma+\int_{-\sigma}^{0}A\left(  \frac{d(\pi \circ\Phi_{s})}{ds}\right)  ds \label{eqn:flowsoln}
\end{equation}
satisfies the hypotheses of Lemma \ref{kde.lem}, on some neighborhood of $(x,0)$ in $T^{\ast}M.$ Thus, $f_{\sigma}$ satisfies (\ref{eqn:preK}) for all $\sigma\in D_{1+\varepsilon}.$
\end{lemma}

\begin{proof}
First, note that $f_{0}$ is identically zero, and hence all of its derivatives are zero, in particular $X(f_{0})=0$ for every vertical vector $X$. Next, to show that $f_{\sigma}$ satisfies the desired differential equation, we first compute that
\begin{equation}
\theta^{A}(X_{E})=2E+A(\pi_{\ast}(X_{E})). \label{eqn:thetaonXE}
\end{equation}
Second, since the integral is along a flow line of $X_{E},$
\[
X_{E}\left(  \int_{-\sigma}^{0}A\left(  \frac{d(\pi\circ\Phi_{s})}{ds}\right) ds\right)  =A(\pi_{\ast}(X_{E}))_{z}-A(\pi_{\ast}(X_{E}))_{\Phi_{-\sigma}(z)}
\]
so that
\begin{align*}
\frac{\partial}{\partial\sigma}\int_{-\sigma}^{0}A\left(  \frac{d(\pi\circ \Phi_{s})}{ds}\right)  ds  &  =A(\pi_{\ast}(X_{E}))_{\Phi_{-\sigma}(z)}\\ 
&  =A(\pi_{\ast}(X_{E}))_{z}-X_{E}\left(  \int_{-\sigma}^{0}A\left(\frac{d(\pi\circ\Phi_{s})}{ds}\right)  ds\right)  .
\end{align*}
Hence,
\begin{align*}
\frac{\partial}{\partial\sigma}f_{\sigma}  &  =E+A(\pi_{\ast}(X_{E}))_{z}-X_{E}\left(  \int_{-\sigma}^{0}A\left(  \frac{d(\pi\circ\Phi_{s})}{ds}\right)  ds\right) \\
&  =-X_{E}\left(  f_{\sigma}\right)  +(\theta^{A}(X_{E})-E),
\end{align*}
where we used (\ref{eqn:thetaonXE}) and $X_{E}(E)=0$ in the last line.
\end{proof}

It now remains only to prove the corollaries of Theorem \ref{thm:hol_sec}.

\bigskip
\begin{proof}
[Proofs of Corollaries \ref{cor:localSection}, \ref{thm:K}, and \ref{thm:suffK}]The first corollary follows from (\ref{eqn:local-hol-cond}), with $\theta$ replaced by $\theta^{A}.$ The second corollary holds because $\bar{\partial}g=0$ if $g$ is holomorphic. For the third corollary, we first note that the condition $\bar{\partial}f_{-i}=(\theta^{A})^{(0,1)}$ implies that $i\partial\bar{\partial}(2if_{-i})=-d\theta^{A}=\omega^{\beta}.$ We then use the fact that the operator $i\partial\bar{\partial}$ maps real-valued functions to real-valued 2-forms. Since $\omega^{\beta}$ is real, it must be the case that $i\partial\bar{\partial}(\operatorname{Re}(2if_{-i}))=\omega^{\beta}$ and $i\partial\bar{\partial}(\operatorname{Im}(2if_{-i}))=0.$ Since also $f_{\sigma}$ is real for real $\sigma,$ we see that $\overline{f_{-i}}=f_{i},$ which leads to the expression for $\operatorname{Re}(2if_{-i}).$
\end{proof}

\section{An antiholomorphic intertwiner}

When $\beta=0,$ the flow $\Phi_{\sigma}$ is just the geodesic flow, which has the following invariance property with respect to time reversal: If $(x(\sigma),p(\sigma))$ is a solution to Hamilton's equations in the $\beta=0$ case, then so is $(x(-\sigma),-p(-\sigma)).$ This property of the flow reflects that the map $(x,p)\mapsto(x,-p)$ is antisymplectic in the $\beta=0$ case. Once $\beta$ is not identically zero, the flow is no longer time-reversal invariant. In the case of a constant positive magnetic field in the plane, for example, the particle always travels in clockwise circles, whereas a time-reversed particle would travel in counterclockwise circles. From the symplectic point of view, time-reversal invariance fails because the map $(x,p)\mapsto(x,-p)$ maps $\omega^{\beta}$ not to $-\omega^{\beta}$ but rather to $-\omega^{-\beta}.$

\begin{theorem}
Let $\nu:T^{\ast}M\rightarrow T^{\ast}M$ be multiplication by $-1$ in the fibers: $\nu(x,p)=(x,-p).$ Let $T_{\beta}^{(1,0)}$ and $T_{\beta}^{(0,1)}$ denote the $(1,0)$ and $(0,1)$ tangent spaces for the magnetic complex structure associated to $\beta.$ Then $\nu$ antiholomorphically intertwines the magnetic adapted complex structures associated to $\beta$ and $-\beta$; that is,
\[
\nu_{\ast}(T_{\beta}^{(1,0)})=T_{-\beta}^{(0,1)}.
\]
\end{theorem}

\begin{proof}
Since we now have two different magnetic fields in the problem, we explicitly denote the dependence various quantities on the magnetic field: $\Phi_{\sigma}^{\beta}$ denotes the Hamiltonian flow associated to the energy function $E$ the symplectic form $\omega^{\beta},$ $X_{E}^{\beta}$ denotes the Hamiltonian vector field generating $\Phi_{\sigma}^{\beta},$ and $P^{\beta}(\sigma)$ denotes the pushforward of the vertical distribution by $\Phi_{\sigma}^{\beta}.$ We now compute the effect of $\nu$ on $X_{E}^{\beta}$ in local coordinates:
\[
\nu_{\ast}(X_{E}^{\beta}(x,p))=p_{j}g^{jl}\frac{\partial}{\partial x^{l}} +\frac{1}{2}\frac{\partial g^{jk}}{\partial x^{l}}p_{j}p_{k}\frac{\partial}{\partial p_{l}}-\beta_{lj}g^{jk}p_{k}\frac{\partial}{\partial p_{l}} =-X_{E}^{-\beta}(x,-p),
\]
which implies
\begin{equation}
\nu\circ\Phi_{\sigma}^{-\beta}\circ\nu=\Phi_{-\sigma}^{\beta}.
\label{eqn:invol}
\end{equation}

From (\ref{eqn:invol}) and the fact that the vertical distribution is
preserved by $\nu$, we obtain
\[
\nu_{\ast}(P^{\beta}(\sigma))=\nu_{\ast}(\Phi_{\sigma}^{\beta})_{\ast}\left(  \nu_{\ast}V\right)=P^{-\beta}(-\sigma).
\]
Analytically continuing this relation, we see that
\[
\nu_{\ast}(P^{\beta}(i))=P^{-\beta}(-i)=\overline{P^{-\beta}(i)},
\]
which proves the theorem.
\end{proof}

\begin{remark}
Using the same reasoning, one can easily show that the map $\Phi_{2\sigma}\circ\nu$ is an antiholomorphic map of the complex structure whose $(1,0)$ distribution is $P^{\beta}(\sigma+i\tau)$ to the complex structure whose $(1,0)$ distribution is $P^{-\beta}(\sigma+i\tau),$ for any $\tau\neq0.$
\end{remark}

\section{The $\mathbb{R}^{2}$ case\label{sec:R2case}}

Consider the constant magnetic field $\beta=B\,dx_{1}\wedge dx_{2}$ on $\mathbb{R}^{2}$ and associated twisted symplectic form
\[
\omega^{\beta}:=dx_{1}\wedge dp_{1}+dx_{2}\wedge dp_{2}-B~dx_{1}\wedge dx_{2}
\]
on $T^{\ast}\mathbb{R}^{2}\simeq\mathbb{R}^{4}=\{(x_{1},x_{2},p_{1},p_{2})\}.$

In this section, we will keep track of the physical constants which appear in various places. Let $\eta$ be a variable with units of frequency. Then the ``complexifier'' function, in the notation of \cite{Thiemann96}, should be
\[
\frac{\text{kinetic energy}}{\eta}=\frac{1}{2m\eta}(p_{1}^{2}+p_{2}^{2}).
\]
That is to say, we choose our metric so that the energy function is $E=\frac{1}{2m\eta}(p_{1}^{2}+p_{2}^{2}).$ The following lemma is easily then verified.

\begin{lemma} The Hamiltonian vector field $X_E$ is
\[
X_{E}=\frac{p_{1}}{m\eta}\frac{\partial}{\partial x_{1}}+\frac{p_{2}}{m\eta}\frac{\partial}{\partial x_{2}}-\tilde{B}\left(p_{1}\frac{\partial}{\partial p_{2}}-p_{2}\frac{\partial}{\partial p_{1}}\right).
\]
The magnetic flow on $\mathbb{R}^{2},$ i.e. the time-$\sigma$ flow of the Hamiltonian vector field $X_{E},$ is
\begin{align*}
\Phi_{\sigma}^{B}(\mathbf{x},\mathbf{p})  &  =\Big(x_{1}+\frac{p_{1}}{B}\sin(\sigma\tilde{B})-\frac{p_{2}}{B}(\cos(\sigma\tilde{B})-1),x_{2} +\frac{p_{2}}{B}\sin(\sigma\tilde{B})+\frac{p_{1}}{B}(\cos(\sigma\tilde {B})-1),\\
&  \qquad p_{1}\cos(\sigma\tilde{B})+p_{2}\sin(\sigma\tilde{B}),-p_{1} \sin(\sigma\tilde{B})~+p_{2}\cos(\sigma\tilde{B})~\Big),
\end{align*}
where $\tilde{B}=\frac{B}{m\eta}$.
\end{lemma}

\begin{lemma}
The magnetic flow $\Phi_{\sigma}$ can be analytically continued to the entire complex plane, as a map of $\mathbb{C}^{4}$ to $\mathbb{C}^{4},$ in particular to $\sigma=i.$ The intersection of $P(i):=(\Phi_{i})_{\ast}V(T^{\ast}\mathbb{R}^{2})$ with its complex conjugate is everywhere $\{0\},$ and thus defines a complex structure on $T^{\ast}\mathbb{R}^{2}.$
\end{lemma}

\begin{proof}
The first part of the lemma is obvious. To show that the intersection of $P(i)$ with $\overline{P(i)}=P(-i)$ is $\{0\}$, first note that $P(i)$ is spanned by $(\Phi_{i})_{\ast}\partial_{p_{j}},~j=1,2$ which are
\begin{align*}
(\Phi_{i})_{\ast}\partial_{p_{1}}  &  =\frac{i\sinh\tilde{B}}{B} \partial_{x_{1}}+\frac{\cosh\tilde{B}-1}{B}\partial_{x_{2}}+\cosh\tilde {B}~\partial_{p_{2}}-i\sinh\tilde{B}~\partial_{p_{2}},\text{ and}\\
(\Phi_{i})_{\ast}\partial_{p_{2}}  &  =-\frac{\cosh\tilde{B}-1}{B} \partial_{x_{1}}+\frac{i\sinh\tilde{B}}{B}\partial_{x_{2}}+i\sinh\tilde
{B}~\partial_{p_{1}}+\cosh\tilde{B}~\partial_{p_{2}}.
\end{align*}
Hence, the intersection $P(i)\cap P(-i)$ will be nonzero if and only if there is a nontrivial linear combination of $(\Phi_{i})_{\ast}\partial_{p_{j}},(\Phi_{-i})_{\ast}\partial_{p_{j}},~j=1,2$ which equals zero. However, the determinant of the matrix whose entries are the coefficients of these vectors may be easily computed to be $-4\sinh^{2}\tilde{B}/B^{2},$ which is never zero.
\end{proof}

\begin{lemma}
\label{lemma:zcoords}The coordinate functions $x_{1}$ and $x_{2}$ composed with the time-$i$ magnetic flow yield the following complex coordinates on $T^*\mathbb{R}^2$
\begin{align*}
z_{1}  &  :=x_{1}+i\frac{\sinh\tilde{B}}{B}p_{1}-\frac{\cosh\tilde{B}-1}{B}p_{2},\\
z_{2}  &  :=x_{2}+i\frac{\sinh\tilde{B}}{B}p_{2}+\frac{\cosh\tilde{B}-1}{B}p_{1}.
\end{align*}
These coordinates can also be obtained via the complexifier formula
\[
z_{j}=e^{iX_{E}}x^{j}.
\]
\end{lemma}

\begin{proof}
The formulas follow easily by using the identities $\sin(ix)=i\sinh x$ and $\cos(ix)=\cosh x$ in the first two components of $\Phi_{i}^{B}$. To apply the complexifier formula and verify the formulas for $z_{1}$ and $z_{2}$, we first compute
\[
X_E(x_{1})=p_{1}/m\eta,~X_E(x_{2})=p_{2}/m\eta,~X_E(p_{1})=Bp_{2}/m\eta,\text{ and }X_E(p_{2})=-Bp_{1}/m\eta.
\]
Now use the Taylor series for the exponential to compute $e^{iX_{E}}
x_{j},~j=1,2.$
\end{proof}

We will compute, using Corollaries \ref{thm:K} and \ref{thm:suffK}, two K\"{a}hler potentials associated to the choice of magnetic potential $A=\frac{B}{2}(x_{1}dx_{2}-x_{2}dx_{1}),$ that is, with respect to the symplectic potential
\begin{equation}
\theta^{A}:=\left(p_{1}-\tfrac{B}{2}x_{2}\right)  dx_{1}+\left(p_{2}+\tfrac{B}{2}x_{1}\right)  dx_{2}. \label{eqn:R2symp-pot}
\end{equation}
The first expression we obtain, by applying Corollary \ref{thm:suffK}, yields a K\"{a}hler potential that is adapted to the above symplectic potential in the sense that $\operatorname{Im}\bar{\partial}\kappa=\theta^{A}.$ The second expression, which we obtain from Corollary \ref{thm:K}, yields a K\"{a}hler potential that is not adapted $\theta^{A}$, but is of independent interest, as it arises in the work of Kr\"{o}tz--Thangavelu--Xu on heat kernel analysis the Heisenberg group \cite{Krotz-Thangavelu-Xu}.

\begin{theorem}
\label{thm:kahlerPots}If $f_{\sigma}$ is as in Theorem \ref{thm:hol_sec} \ref{item:hol_sec2}, then there exists a holomorphic function $g$ on $T^{\ast}\mathbb{R}^{2}$ such that the function $\kappa_{1}:=2if_{-i}+g$ is real valued and thus is a K\"{a}hler potential adapted to the symplectic potential $\theta^{A}$ in (\ref{eqn:R2symp-pot}). Explicitly, we have
\begin{equation}
\kappa_{1}:=-B(uy-vx)+B\coth\tilde{B}~(v^{2}+y^{2})+\frac{B}{2}\tanh\frac{\tilde{B}}{2}~(x^{2}-y^{2}+u^{2}-v^{2}),
\end{equation}
where $z_{1}=x+iy$ and $z_{2}=u+iv$ and where $\tilde{B}=\tfrac{B}{m\eta}.$

Alternatively, we may consider the function $\kappa_{2}:=\operatorname{Re}(2if_{-i}),$ which is also a K\"{a}hler potential but not adapted to $\theta^{A}.$ Explicitly, we have
\[
\kappa_{2}=-B~(uy-vx)+B\coth\frac{B}{m\eta}(v^{2}+y^{2}).
\]
\end{theorem}

Before turning to the proof of this result, we describe the connection of the K\"{a}hler potential $\kappa_{2}$ with the work of Kr\"{o}tz--Thangavelu--Xu \cite{Krotz-Thangavelu-Xu}. It follows from the discussion at the beginning of Section \ref{sec:K} that if $L$ is a globally trivializable $\omega^{\beta}$-prequantum bundle over $T^{\ast}\mathbb{R}^{2}$, then the Hermitian product on the space of holomorphic sections can be written (up to an overall constant, which we ignore for the purposes of this discussion) in terms of the associated holomorphic functions $\psi_{j},~j=1,2$ as
\begin{equation}
\left\langle s_{1},s_{2}\right\rangle =\int_{M}\bar{\psi}_{1}\psi_{2}e^{-\kappa}~(\omega^{\beta})^{\wedge n}, \label{eqn:hl2quant}
\end{equation}
where $\kappa$ is a global K\"{a}hler potential associated to the trivialization of $L$.

In \cite{Krotz-Thangavelu-Xu}, the authors study the space of holomorphic functions $\psi$ on $\mathbb{C}^{2}$ such that
\begin{equation}
\int_{\mathbb{C}^{2}}\left\vert \psi(z)\right\vert ^{2}e^{\lambda (uy-vx)-\lambda\coth(2\lambda t)(v^{2}+y^{2})}d^{4}z<\infty, \label{eqn:KTX}
\end{equation}
where $z_{1}=x+iy$ and $z_{2}=u+iv.$ (In the notation of \cite[Eq. (4.1.2),(4.2.2)]{Krotz-Thangavelu-Xu}, this space is $\mathcal{B}_{t}^{\lambda}$). We can see that the space in (\ref{eqn:KTX}) coincides with the space in (\ref{eqn:hl2quant}) provided that we use the K\"{a}hler potential $\kappa_{2}$ and make the identifications:
\begin{align*}
B  &  =\lambda\\
m\eta &  =\frac{1}{2t}.
\end{align*}
Hence, the holomorphic $L^{2}$ space of \cite{Krotz-Thangavelu-Xu} determined by (\ref{eqn:KTX}) can be interpreted as the space of square-integrable holomorphic sections of the prequantum bundle on $(T^{\ast}\mathbb{R} ^{2},\omega+\pi^{\ast}\beta)$ with respect to the magnetic complex structure induced by $\beta=\lambda~dx_{1}\wedge dx_{2}$.

The rest of this section is devoted to the proof of Theorem \ref{thm:kahlerPots}.

\begin{lemma}
\label{lemma:fonR2}The function $f_{\sigma}$ in Theorem \ref{thm:hol_sec} \ref{item:hol_sec2} is
\[
f_{\sigma}=-\frac{\sin\sigma\tilde{B}}{2}(x_{2}p_{1}-x_{1}p_{2}) -\frac{\cos\sigma\tilde{B}-1}{2}(x_{1}p_{1}+x_{2}p_{2})+\frac{\sin\sigma\tilde{B}}{2B}(p_{1}^{2}+p_{2}^{2}),
\]
where $\tilde{B}=\frac{B}{m\eta}.$
\end{lemma}

\begin{proof}
We compute
\begin{align*}
f_{\sigma}  &  =\frac{\sigma}{2m\eta}(p_{1}^{2}+p_{2}^{2})-\frac{B}{2}\int_{0}^{-\sigma}\left(  x_{1}+\frac{p_{1}}{B}\sin t\tilde{B}-\frac{p_{2}}{B}(\cos t\tilde{B}-1)\right)  \left(  p_{2}\cos t\tilde{B}-p_{1}\sin t\tilde{B}\right) \\
&  \qquad\qquad\qquad\qquad-\left(  x_{2}+\frac{p_{2}}{B}\sin t\tilde{B} +\frac{p_{1}}{B}(\cos t\tilde{B}-1)\right)  \left(  p_{1}\cos t\tilde{B} +p_{2}\sin t\tilde{B}\right)  dt
\end{align*}
which upon integration yields the desired expression.
\end{proof}

The function $f_{\sigma}$ can be analytically continued to the entire complex plane, so we only need to find a holomorphic function $g$ such that $2if_{-i}+g$ is real valued.

\begin{lemma}
\label{lemma:kp} Let $g=B\tanh\left(  \tilde{B}/2\right)  (z_{1}^{2}+z_{2}^{2})$. Then $2if_{-i}+g$ is real valued. In terms of the real and imaginary parts of $z_{1}=x+iy$ and $z_{2}=u+iv$, the result is a K\"{a}hler potential $\kappa_1:=2if_{-i}+g$ given by
\begin{equation}
\kappa_1=-B(uy-vx)+B\coth\tilde{B}~(v^{2}+y^{2})+B\tanh\frac{\tilde{B}}{2}~(x^{2}-y^{2}+u^{2}-v^{2}) \label{eqn:K}
\end{equation}
\end{lemma}

\begin{proof}
From Lemma \ref{lemma:fonR2} we see that
\[
2if_{-i}=-i(\cosh\tilde{B}-1)(p_{1}x_{1}+p_{2}x_{2})-\sinh\tilde{B} ~(p_{1}x_{2}-p_{2}x_{1})+\frac{\sinh\tilde{B}}{B}(p_{1}^{2}+p_{2}^{2}).
\]
Lemma \ref{lemma:zcoords} yields
\begin{align}
p_{1}x_{1}+p_{2}x_{2}  &  =\frac{B}{\sinh\tilde{B}}(uv+xy),\nonumber\\
p_{1}x_{2}-p_{2}x_{1}  &  =\frac{B}{\sinh\tilde{B}}(uy-vx)-\frac{B}
{\sinh\tilde{B}}\frac{\cosh\tilde{B}-1}{\sinh\tilde{B}}(v^{2}+y^{2}),\text{and}\label{eqn:xpuv}\\
p_{1}^{2}+p_{2}^{2}  &  =\frac{B^{2}}{\sinh^{2}\tilde{B}}(v^{2}+y^{2}).\nonumber
\end{align}
We thus obtain
\[
2if_{-i}=-B~(uy-vx)+B\coth\tilde{B}~(v^{2}+y^{2})-iB\tanh\tfrac{\tilde{B}}{2}(uv+xy).
\]
Adding $g$ to this yields the desired expression.
\end{proof}

To obtain our second K\"{a}hler potential $\kappa_{2},$ we simply take the real part of $2if_{-i},$ using Lemma \ref{lemma:fonR2}. Note that a holomorphic trivializing section $s_{0}$ induces a K\"{a}hler potential $\kappa =-\log\left\vert s_{0}\right\vert ^{2},$ and a change of trivialization $s_{0}\mapsto e^{G}s_{0}$ induces a corresponding change of K\"{a}hler potential $\kappa\mapsto\kappa-2\operatorname{Re}G.$ Our new K\"{a}hler potential can be obtained from the K\"{a}hler potential given in Theorem \ref{lemma:kp} via change of trivialization $s_{0}\mapsto e^{G}s_{0}$ with $G=\frac{\lambda}{4}\tanh\left(\frac{\lambda}{m\eta}\right)  \left((x+iy)^{2}+(u+iv)^{2}\right)$.

\section{The $S^{2}$ case\label{sec:S2case}}

We consider the sphere $S^{2}$ of radius $r>0$ inside $\mathbb{R}^{3},$ along with its tangent bundle $TS^{2},$ as follows:
\begin{align*}
S^{2}  &  =\left\{  \mathbf{x}\in\mathbb{R}^{3}\left\vert x^{2}=r^{2}\right.\right\}  ,\\
TS^{2}  &  =\left\{  (\mathbf{x},\mathbf{p})\in\mathbb{R}^{3}\times \mathbb{R}^{3}\left\vert x^{2} =r^{2},~\mathbf{x}\cdot\mathbf{p}=0\right.
\right\}  ,
\end{align*}
where $x^{2}=x_{1}^{2}+x_{2}^{2}+x_{3}^{2}.$ We use on $S^{2}$ the Riemannian metric inherited from the standard metric on $\mathbb{R}^{3}$ and we use this metric to identify the tangent and cotangent bundles of $S^{2}.$

We also consider the complex sphere $S_{\mathbb{C}}^{2}$ and its tangent bundle as follows:
\begin{align*}
S_{\mathbb{C}}^{2}  &  =\left\{  \mathbf{x}\in\mathbb{C}^{3}\left\vert x^{2}=r^{2}\right.  \right\} \\
TS_{\mathbb{C}}^{2}  &  =\left\{  (\mathbf{x},\mathbf{p})\in\mathbb{C}^{3}\times\mathbb{C}^{3}\left\vert x^{2}=r^{2},~\mathbf{x}\cdot\mathbf{p}=0\right.  \right\}  .
\end{align*}
Note that $TS_{\mathbb{C}}^{2}$ is a complex manifold of dimension 4 and that $TS^{2}$ sits inside $TS_{\mathbb{C}}^{2}$ as a totally real submanifold of maximal dimension.

We consider the rotationally invariant 2-form $\beta$ on $S^{2}$ given by 
\[
\beta=\ \frac{B}{r}\left(  x_{1}dx_{2}\wedge dx_{3}+x_{2}dx_{3}\wedge dx_{1}+x_{3}dx_{1}\wedge dx_{2}\right),
\]
where $B$ is a fixed real constant. As in the rest of the paper, we let $\omega^{\beta}=\omega-\pi^{\ast}\beta,$ where $\omega$ is the canonical 2-form on the (co)tangent bundle, and $\Phi_{\sigma}$ is the Hamiltonian flow generated from the kinetic energy function $E(\mathbf{x},\mathbf{p}):=\left\vert \mathbf{p}\right\vert ^{2}/2$ by means of the symplectic form $\omega^{\beta}.$

\begin{theorem}
\label{aForm.thm}For each $\sigma\in\mathbb{R},$ let $a_{\sigma}$ be the $\mathbb{R}^{3}$-valued function on $TS^{2}$ given by
\[
\mathbf{a}_{\sigma}=\pi\circ\Phi_{\sigma},
\]
where $\pi:TS^{2}\rightarrow S^{2}$ is the projection. Then for each $(\mathbf{x},\mathbf{p})\in TS^{2},$ the function $\sigma\mapsto \mathbf{a}_{\sigma}(\mathbf{x},\mathbf{p})$ admits an extension to a holomorphic, $\mathbb{C}^{3}$-valued function on all of $\mathbb{C},$ which we continue to denote by $\mathbf{a}_{\sigma}.$ The function $\mathbf{a}:=\mathbf{a}_{i}$ is a diffeomorphism of $TS^{2}$ onto the complex sphere $S_{\mathbb{C}}^{2}$ and the pullback of the complex structure on $S_{\mathbb{C}}^{2}$ is a magnetic adapted complex structure defined on all of $TS^{2}$ that is everywhere K\"{a}hler with respect to $\omega^{\beta}.$

The function $\mathbf{a}$ may be computed explicitly as
\[
\mathbf{a}(\mathbf{x,p})=(\cosh L)\mathbf{x}+i\frac{\sinh L}{L}\mathbf{p}-\frac{\left(  \cosh L-1\right)  }{L^{2}}B\frac{\mathbf{J}(\mathbf{x},\mathbf{p})}{r},
\]
where
\[
L=\frac{\sqrt{p^{2}+r^{2}B^{2}}}{r}
\]
and where $\mathbf{J}$ is the moment mapping for the action of $\mathrm{SO}(3)$ on $(TS^{2},\omega^{\beta}),$ given explicitly as
\[
\mathbf{J}(\mathbf{x},\mathbf{p})=\mathbf{x}\times\mathbf{p}-rB\mathbf{x}.
\]
\end{theorem}

Note that the components of the vector-valued function $\mathbf{a(x,p)}$ are of the form $f\circ\pi\circ\Phi_{i}$ where $f$ are the components of the vector $\mathbf{x}$, so that, as per the method of Thiemann, the $a_j$'s are the holomorphic extensions to $T^*S^2$ of the $x_j$'s.

\begin{lemma}
Let $\{\xi_{1},\xi_{2},\xi_{3}\}$ be the basis for $\mathrm{so}(3)$ given by $(\xi_{j})_{kl}=\varepsilon_{jkl},$ where $\varepsilon_{jkl}$ is the totally antisymmetric tensor. Let $\mathrm{so}(3)$ and $\mathrm{so}(3)^{\ast}$ be identified with $\mathbb{R}^{3}$ by means of this basis. Then the function $\mathbf{J}:TS^{2}\rightarrow\mathbb{R}^{3} \cong\mathrm{so}(3)^{\ast}$ given by
\[
\mathbf{J}(\mathbf{x},\mathbf{p})=\mathbf{x}\times\mathbf{p}-rB\mathbf{x}
\]
is an equivariant moment map for the action of $\mathrm{SO}(3)$ on $(TS^{2},\omega^{\beta}).$
\end{lemma}

\begin{proof}
If $\mathrm{SO}(3)$ acts on $\mathbb{R}^{3}\times\mathbb{R}^{3}$ by simultaneous rotations of $\mathbf{x}$ and $\mathbf{p},$ then this action preserves $TS^{2}$ and leaves invariant the symplectic form $\omega^{\beta}$ on $TS^{2}.$ The infinitesimal action of $\mathrm{so}(3)$ on $TS^{2}$ is given by vector fields $\hat{\xi}_{j},$ where, for example,
\begin{equation}
\hat{\xi}_{1}=x_{2}\frac{\partial}{\partial x_{3}}-x_{3}\frac{\partial}{\partial x_{2}}+p_{2}\frac{\partial}{\partial p_{3}}-p_{3}\frac{\partial}{\partial p_{2}}. \label{x1def}
\end{equation}
This vector field, defined initially on $\mathbb{R}^{3}\times\mathbb{R}^{3},$ is tangent to $TS^{2}$ at each point in $TS^{2}.$ We may obtain the expressions for $\hat{\xi}_{2}$ and $\hat{\xi}_{3}$ by cyclic permutation of the indices in (\ref{x1def}).

To obtain a moment map, we wish to find a function $J_{j}$ such that
\[
\omega^{\beta}(\hat{\xi}_{j},\cdot)=dJ_{j}.
\]
Of course, $J_{j}$ depends also on $\beta,$ but we suppress this dependence in the notation. Since $\omega^{\beta}$ is a sum of two terms, $J_{j}$ can be taken as a sum of the usual angular momentum and another function, which depends only on $\mathbf{x}.$ It is not hard to verify that the expression
\[
\mathbf{J}(\mathbf{x},\mathbf{p})=\mathbf{x}\times\mathbf{p}-rB\mathbf{x}
\]
is an equivariant moment map, where $\mathbf{J}(\cdot,\cdot)$ takes values in $\mathbb{R}^{3},$ which is identified with $\mathrm{so}(3)^{\ast}.$ To verify this expression, it suffices to check that
\[
\beta(\xi_{j},\cdot)=\frac{B}{r}\left(  r^{2}dx_{j}-x_{j}(x_{1}dx_{1}+x_{2}dx_{2}+x_{3}dx_{3}\right)  )=Br~dx_{j},
\]
since $d(x_{1}^{2}+x_{2}^{2}+x_{3}^{2})=0$ on $TS^{2}.$
\end{proof}

\begin{lemma}
The Hamiltonian flow $\Phi_{\sigma}$ may be computed as
\begin{equation}
\Phi_{\sigma}(\mathbf{x},\mathbf{p})=\left(  \exp\left[  \frac{\sigma}{r^{2}}\mathbf{J}(\mathbf{x},\mathbf{p})\cdot\mathbf{\xi}\right]  (\mathbf{x}),\exp\left[  \frac{\sigma}{r^{2}}\mathbf{J}(\mathbf{x},\mathbf{p})\cdot\mathbf{\xi}\right]  (\mathbf{p})\right)  . \label{flow}
\end{equation}
The notation in (\ref{flow}) is as follows: For each fixed point $(\mathbf{x},\mathbf{p})\in TS^{2},$ the expression $\mathbf{J}(\mathbf{x},\mathbf{p})\cdot\mathbf{\xi}$ denotes the Lie algebra element 
\[
\mathbf{J}(\mathbf{x},\mathbf{p})\cdot\mathbf{\xi}=J_{1}(\mathbf{x} ,\mathbf{p})\xi_{1}+J_{2}(\mathbf{x},\mathbf{p})\xi_{2}+J_{3}(\mathbf{x}
,\mathbf{p})\xi_{3}
\]
and the expression ``exp'' denotes the exponential map from the Lie algebra $\mathrm{so}(3)$ to the group $\mathrm{SO}(3).$
\end{lemma}

\begin{proof}
Note that the two terms in the definition of $\mathbf{J}$ are orthogonal, and also that $\mathbf{x}$ and $\mathbf{p}$ are orthogonal at each point of $TS^{2}.$ Thus
\begin{equation}
J^{2}(\mathbf{x},\mathbf{p})=r^{2}p^{2}+r^{4}B^{2} \label{jSquared}
\end{equation}
at each point $(\mathbf{x},\mathbf{p})\in TS^{2},$ where $a^{2}$ denotes the squared magnitude of a vector $\mathbf{a}.$ As the Poisson bracket of two functions $f$ and $g$ is given by $\{f,g\}:=X_f(g)$, we may compute \begin{equation}
\left\{  E,f\right\}  =\frac{1}{2r^{2}}\left\{  J^{2},f\right\}  , \label{J2E}
\end{equation}
where $E(\mathbf{x},\mathbf{p}):=p^{2}/2$ is the energy function. Now, it follows from the definition of the moment map that
\begin{equation}
\{J_{j},f\}=\hat{\xi}_{j}f. \label{bracketJj}
\end{equation}
Thus,
\begin{equation}
\{J^{2},f\}=2J_{j}(\hat{\xi}_{j}f). \label{bracketJ2}
\end{equation}
Furthermore, since $J^{2}$ is invariant under simultaneous rotations of $\mathbf{x}$ and $\mathbf{p},$
\begin{equation}
\{J^{2},J_{j}\}=-\hat{\xi}_{j}J^{2}=0. \label{bracketJJ}
\end{equation}

The correctness of (\ref{flow}) now follows from direct differentiation and the fact---which follows from (\ref{bracketJ2}) and (\ref{bracketJJ})---that the map on the right-hand side of (\ref{flow}) preserves the value of $\mathbf{J}$.
\end{proof}

We these two lemmas in hand, we are ready for the proof of Theorem \ref{aForm.thm}.

\bigskip
\begin{proof}
[Proof of Theorem \ref{aForm.thm}, Part 1]The function $\mathbf{J}$ extends to a holomorphic function on $TS_{\mathbb{C}}^{2},$ given by the same formula. Furthermore, the complex manifold $TS_{\mathbb{C}}^{2}$ is invariant under the action of the complex rotation group $\mathrm{SO}(3;\mathbb{C}).$ Thus, for each $\sigma\in\mathbb{C},$ we can define a holomorphic map $\Phi_{\sigma}:TS_{\mathbb{C}}^{2}\rightarrow TS_{\mathbb{C}}^{2}$ by using the same formula as in (\ref{flow}). The family of maps $\Phi_{\sigma}$ is a holomorphic flow on $TS_{\mathbb{C}}^{2}.$

Observe now that for each $(\mathbf{x},\mathbf{p})\in TS^{2},$ we have a natural identification 
\begin{equation}
\left[  T_{(\mathbf{x},\mathbf{p})}S^{2}\right]  _{\mathbb{C}}\cong T_{(\mathbf{x},\mathbf{p})}\left[  TS_{\mathbb{C}}^{2}\right].
\label{tangentIdent}
\end{equation}
(The right-hand side of (\ref{tangentIdent}) is a complex subspace of $\mathbb{C}^{6}$ of dimension 4, containing $T_{(\mathbf{x},\mathbf{p})}S^{2}$ as a totally real subspace of dimension 4.) We may therefore define subspaces of the complexified tangent bundle by the formula (\ref{eqn:mag-geod_flow}). This family of subspaces depends holomorphically on $\sigma$, so the desired analytic continuation exists. It remains to show that the subspaces are disjoint from their complex conjugates.

Consider the family of functions $\mathbf{a}_{\sigma}$ defined in the theorem. Since $\Phi_{\sigma}$ now makes sense for all $\sigma\in\mathbb{C},$ we can use the same formula to define the holomorphic extension. Since the derivatives of $\mathbf{a}_{0}(\mathbf{x},\mathbf{p})=\mathbf{x}$ in the vertical directions are zero, the derivatives of $\mathbf{a}_{\sigma}$ in the directions of $P_{z}(-\sigma)$ will also be zero, for all $\sigma\in \mathbb{R}.$ Since, by construction, the family $P_{z}(\sigma)$ depends holomorphically on $\sigma,$ we see that the derivatives of $\mathbf{a}_{i}$ in the directions of $P_{z}(-i)=\overline{P_{z}(i)}$ will also be zero. We will show below that the function $\mathbf{a}:=\mathbf{a}_{i}$ maps $TS^{2}$ diffeomorphically onto $S_{\mathbb{C}}^{2}.$ Now, over on $S_{\mathbb{C}}^{2},$ derivatives of the function $f(\mathbf{a}):=\mathbf{a}$ in the direction of a nonzero vector in the $(1,0)$ direction are never zero, whereas derivatives in the $(0,1)$ directions are always zero. Thus, the \textit{only} directions in which $f(\mathbf{a})=\mathbf{a}$ has derivative zero is in the $(0,1)$ directions. Since $\overline{P_{z}(i)}$ is a 2-dimensional subspace of directions in which $\mathbf{a}$ is constant, $\overline{P_{z}(i)}$ must coincide with the pullback of the $(0,1)$ subspace under the map $\mathbf{a}.$ It follows that $P_{z}(i)$ is the pullback of the $(1,0)$ subspace and thus, since $\mathbf{a}$ is a diffeomorphism, that $P_{z}(i)$ and $\overline{P_{z}(i)}$ intersect only at 0.

By the general construction, the subspaces $P_{z}(i)$ are positive complex Lagrangian subspaces on the zero-section, and since at any point $z$, the subspace $P_{z}(i)$ intersects $\overline{P_{z}(i)}$ only at $0$, it follows that $P_{z}(i)$ is a positive complex Lagrangian subspace for every $z\in TS^{2},$ whence $P(i)$ determines a K\"{a}hler complex structure on all of $TS^{2}.$

For any vector $\mathbf{J}\in\mathbb{R}^{3},$ the Lie algebra element $\mathbf{J}\cdot\mathbf{\xi},$ viewed as a skew-symmetric linear operator on $\mathbb{R}^{3},$ is equal to the map $\mathbf{v}\mapsto\mathbf{J} \times\mathbf{v},$ where $\times$ is the cross product. Thus, we need to apply $\mathbf{J}\times\cdot$ repeatedly to the vector $\mathbf{x}.$ Now, $\mathbf{J}\times\mathbf{x}=-\mathbf{x}\times(\mathbf{x}\times\mathbf{p})=r^{2}\mathbf{p}.$ Furthermore, since $\mathbf{J}$ and $\mathbf{p}$ are orthogonal, we have $\mathbf{J}\times(\mathbf{J}\times\mathbf{p})=-J^{2}\mathbf{p}.$ Now, $\mathbf{J}$ and $\mathbf{x}$ are not orthogonal unless $B=0,$ but $\mathbf{J}$ and $\mathbf{J}\times\mathbf{p}$ are orthogonal, so $\mathbf{J}\times(\mathbf{J}\times(\mathbf{J}\times\mathbf{p}))=-J^{2}\mathbf{J}\times\mathbf{p}.$ Thus, when applying powers of $\mathbf{J}\times\cdot$ to the vector $\mathbf{x},$ all the terms after the zeroth will alternate between a multiple of $\mathbf{p}$ and a multiple of $\mathbf{J}\times\mathbf{p}.$ We get, then,
\[
\exp\left[  \frac{i}{r^{2}}\mathbf{J}\times\cdot\right]  (\mathbf{x})=\mathbf{x+}\frac{r^{2}}{J^{2}}\mathbf{J}\times\mathbf{p}-\cosh\left(\frac{J}{r^{2}}\right)  \frac{r^{2}}{J^{2}}\mathbf{J}\times\mathbf{p}+i\sinh\left(  \frac{J}{r^{2}}\right)  \frac{r^{2}}{J}\mathbf{p}.
\]
Recognizing that $J/r^{2}=L$ and simplifying gives the expression for $\mathbf{a}$ in the theorem.
\end{proof}

\bigskip
\begin{proof}
[Proof of Theorem \ref{aForm.thm}, Part 2]It remains only to show that the map $(\mathbf{x},\mathbf{p})\mapsto\mathbf{a}(\mathbf{x},\mathbf{p})$ is a diffeomorphism of $TS^{2}$ onto $S_{\mathbb{C}}^{2}.$

\textbf{Injectivity.} We need to show that we can uniquely recover the pair $(\mathbf{x},\mathbf{p})$ from $\mathbf{a}(\mathbf{x},\mathbf{p}).$ If $\mathbf{p}=0$, then $\mathbf{a}=\mathbf{x},$ so $\mathbf{a}$ is injective on this set. Furthermore, if $\mathbf{p}\neq0,$ then $\operatorname{Im} \mathbf{a}\neq0,$ so there is no overlap between the $\mathbf{p}=0$ case and the $\mathbf{p}\neq0$ case.

We now establish injectivity of $\mathbf{a}(\mathbf{x},\mathbf{p})$ in the $\mathbf{p}\neq0$ case. We start by observing that
\begin{equation}
\left\vert \operatorname{Im}\mathbf{a}\right\vert ^{2}=\left(  \frac {\sinh\left(  p^{2}+r^{2}B^{2}\right)}{p^{2}+r^{2}B^{2}}\right)  ^{2}p^{2}.
\label{ima}
\end{equation}
The right-hand side of (\ref{ima}) is a monotonic function of $p,$ for $p\geq0,$ which equals 0 when $p=0$ and which tends to $+\infty$ as $p$ tends to $+\infty.$ Thus, we can recover $p=\left\vert \mathbf{p}\right\vert $---and hence also $J,$ by (\ref{jSquared})---from $\mathbf{a}.$ Since, also, $\operatorname{Im}\mathbf{a}$ is a positive multiple of $\mathbf{p},$ we can recover $\mathbf{p}$ itself from $\mathbf{a}.$

Now, given any point $(\mathbf{x},\mathbf{p})\in TS^{2}$ with $\mathbf{p}\neq0,$ it is easy to see that the vectors $\operatorname{Im}\mathbf{a}(\mathbf{x},\mathbf{p}),$ $\operatorname{Re}\mathbf{a}(\mathbf{x},\mathbf{p}),$ and $\mathbf{p}\times\operatorname{Re}\mathbf{a}(\mathbf{x},\mathbf{p})$ form an orthogonal basis for $\mathbb{R}^{3}.$ If we take the dot product of $\mathbf{x}$ with each of these basis elements, we obtain, after a little algebra:
\begin{align*}
(\operatorname{Im}\mathbf{a)}\cdot\mathbf{x}  &  =0\\
(\operatorname{Re}\mathbf{a})\cdot\mathbf{x}  &  =r^{2}\left(  \frac{r^{4}B^{2}}{J^{2}}+\cosh\left(  \frac{J}{r^{2}}\right)  \frac{J^{2} -r^{4}B^{2}}{J^{2}}\right)\\
(\mathbf{p}\times\operatorname{Re}\mathbf{a)}\cdot\mathbf{x}  &  =-r^{2}p^{2}\frac{r^{3}B}{J^{2}}\left(  1-\cosh\left(  \frac{J}{r^{2}}\right) \right),
\end{align*}
These conditions uniquely determine $\mathbf{x},$ showing that $\mathbf{x}$ as well as $\mathbf{p}$ is determined by $\mathbf{a}.$

\textbf{Surjectivity.} Fix some $\mathbf{b}\in S_{\mathbb{C}}^{2}.$ If $\operatorname{Im}\mathbf{b}=0,$ then $\left\vert \operatorname{Re} \mathbf{b}\right\vert =r$ and so we can take $\mathbf{p}=0$ and $\mathbf{x} =\mathbf{b}.$ If $\operatorname{Im}\mathbf{b}\neq0,$ then we can exploit the $SO(3)$-equivariance of $\mathbf{a}(\cdot,\cdot)$ to assume, without loss of generality, that $\operatorname{Im}\mathbf{b}$ is a positive multiple of $\mathbf{e}_{3},$ so that $\operatorname{Re}\mathbf{b}$ lies in the $(x_{1},x_{2})$ plane. We can then find a positive number $\alpha$ such that $\operatorname{Im}\mathbf{a}(\mathbf{x},\alpha\mathbf{e}_{3})=\operatorname{Im}\mathbf{a}$ for all vectors $\mathbf{x}$ of length $r$ lying in the $(x_{1},x_{2})$ plane (compare (\ref{ima})). Meanwhile, $\operatorname{Re}\mathbf{a(}r\mathbf{e}_{1},\alpha\mathbf{e}_{3})$ will be a vector whose magnitude equals $\left\vert \operatorname{Re}\mathbf{b}\right\vert $ and that also lies in the $(x_{1},x_{2})$ plane. Thus, there is some rotation $R$ in the $(x_{1},x_{2})$ plane that rotates $\operatorname{Re}\mathbf{a(}r\mathbf{e}_{1},\alpha\mathbf{e}_{3})$ to $\operatorname{Re}\mathbf{a},$ in which case $(rR\mathbf{e}_{1},\alpha\mathbf{e}_{3})$ is the vector we want.

\textbf{Diffeomorphism.} We have established that $\mathbf{a}$ maps $TS^{2}$  injectively onto $S_{\mathbb{C}}^{2}.$ To show that $\mathbf{a}$ is a diffeomorphism, it suffices to show that the differential of $\mathbf{a}$ is injective at each point $(\mathbf{x},\mathbf{p})$ in $TS^{2}.$ By the $SO(3)$-equivariance of $\mathbf{a},$ it suffices to compute the differential at a point with $\mathbf{x}=r \mathbf{e}_{1}$ and with $\mathbf{p}$ equal to a non-negative multiple of $\mathbf{e}_{2}.$ We now compute the derivative of $\mathbf{a}$ along each of the following four curves: (1) A curve in which we rotate $\mathbf{x}$ in the $(x_{1},x_{3})$ plane while keeping $\mathbf{p}$ fixed, (2) A curve in which we rotate both $\mathbf{x}$ and $\mathbf{p}$ in the $(x_{1},x_{2})$ plane, (3) A curve in which we keep $\mathbf{x}$ fixed and vary $\mathbf{p}$ in the $\mathbf{e}_{3}$ direction, and (4) A curve in which we keep $\mathbf{x}$ fixed and vary $\mathbf{p}$ in the $\mathbf{e}_{2}$ direction. We compute the derivatives of $\mathbf{a}$ along such curves as vectors in $\mathbb{C}^{3}=\mathbb{R}^{6},$ using the real basis $(\mathbf{e}_{1},\mathbf{e}_{2},\mathbf{e}_{3},i\mathbf{e}_{1},i\mathbf{e} _{2},i\mathbf{e}_{3}).$ In the first three cases, the magnitude $p$ of $\mathbf{p}$ is infinitesimally constant along our curve, which simplifies the computation of the derivative.

Substituting the formula for $\mathbf{J}$ into the expression in Theorem \ref{aForm.thm}, we obtain an expression of the form $\mathbf{a}(\mathbf{x},\mathbf{p})=\alpha (p)\mathbf{x}+i\beta (p)\mathbf{p}-\gamma (p)\mathbf{x}\times \mathbf{p}.$ The vectors resulting from the derivatives in the preceding paragraph are then as follows:

\begin{equation*}
\left(
\begin{array}{c}
rp\gamma (p) \\
0 \\
r\alpha (p) \\
0 \\
0 \\
0
\end{array}
\right) ;\quad \left(
\begin{array}{c}
0 \\
r\alpha (p) \\
0 \\
-p\beta (p) \\
0 \\
0
\end{array}
\right) ;\quad \left(
\begin{array}{c}
0 \\
r\gamma (p) \\
0 \\
0 \\
0 \\
\beta (p)
\end{array}
\right) ;\quad \left(
\begin{array}{c}
r\alpha ^{\prime }(p) \\
0 \\
-(r\gamma (p)+rp\gamma ^{\prime }(p)) \\
0 \\
\beta (p)+p\beta ^{\prime }(p) \\
0
\end{array}
\right) .
\end{equation*}
If we can show that these vectors are linearly independent, then $\mathbf{a}(\cdot ,\cdot )$ will be a local diffeomorphism. To show independence, we project the vectors into $\mathbb{R}^{4}$ by throwing away the first and fourth entries in each vector. It is easy to see that the resulting vectors in $\mathbb{R}^{4}$ are independent, once we verify that $\alpha (p),$ $\beta (p),$ and $p\beta ^{\prime }(p)$ are all positive for all $p\geq 0.$ 
\end{proof}

\end{document}